\documentclass[11pt]{article}
\usepackage[latin1]{inputenc}
\usepackage{amsmath, amssymb, amsfonts, amsthm, amssymb, amscd, mathtools, mathrsfs}
\usepackage{color}
\usepackage{multicol}
\usepackage{ulem, stmaryrd}
\usepackage{latexsym, epsfig}
\usepackage{amssymb, amsmath, amsfonts, amscd}
\usepackage{exscale}
\usepackage{ifthen}
\usepackage{longtable}
\usepackage{subfigure}
\usepackage{cite}
\usepackage{lscape}
\usepackage{tikz}
\usetikzlibrary{calc}
\usepackage{picture}

\newtheorem{theo}{\textbf{Theorem}\ }
[section]
\newtheorem{lemma}[theo]{\textbf{Lemma}\ }

\newtheorem{prop}[theo]{\textbf{Proposition}\ }

\newtheorem{notation}[theo]{Notation\ }

\newtheorem{example}[theo]{\textbf{Example}\ }
\newtheorem{examples}[theo]{\textbf{Examples}\ }

\def \bE{\mathbb{E}}

\def \bZ{\mathbb Z}
\def \bR{\mathbb{R}}

\def \bP{\mathbb{P}}
\def \bS{\mathbb{S}}

 \def \cC{\mathcal{C}}
  \def \cF{\mathcal{F}}

\def \cG{\mathcal{G}}

\def \by{\bf y}

\begin{document}
\title{
A local limit theorem for nonlattice multidimensional random walks in cones
}
\author{
D.C. Pham\footnote{CERADE ESAIP,  18 rue du 8 mai 1945 - CS 80022 - 49180 St-Barth\'elemy d'Anjou.
			LAREMA,  UMR CNRS 6093,  Universit\'e d'Angers,  France.   email: dpham@esaip.org}\,  \ 
			M. Peign\'e\footnote{Email: Institut Denis Poisson UMR 7013,   Universit\'e de Tours,  Universit\'e d'Orl\'eans,  CNRS  France.
marc.peigne@univ-tours.fr}
\quad \&  \quad 
D. T.  Son \footnote{
	 Institute of Mathematics,  Vietnam Academy of Science and Technology,  Vietnam. dtson@math.ac.vn}
}
\date{\today}
\maketitle

\begin{abstract}
   We study the asymptotic behavior of a   nonlattice random walk in a general cone of $\mathbb R^d$. Following the approach initiated by D. Denisov and V. Wachtel in \cite{DW2015},  we use a strong approximation of random walks by the Brownian motion   and prove local limit theorems,   combining integral theorems for random walks in cones with classical  theorems for unrestricted random walks. 
\end{abstract}

\noindent Keywords:   random walks,  first exit time,    theory of fluctuations

\vspace{5mm}

\noindent AMS classification   60J80,  60F17,  60K37.

\vspace{5mm}


\section{Introduction}
 The theory of random walks conditioned to stay in cones  is a popular study item  for a few decades,  which appears in many situations falling in various areas  connected to probability theory and combinatorics. We may mention for instance nonintersecting paths,  and their connections with Young diagram and different physical modeling \cite{Fu},  random walks in the quarter-plane with reflection at the boundary \cite{FIM},  random walks in Weyl chambers \cite{GZ} and their connections with random walks on Lie groups \cite{Var} or branching processes in random environment \cite{GK},  \cite{LPP}.

The main purpose of the present paper is to extend to nonlattice random walks  a recent and important result concerning discrete random walks,  due to D. Denisov and V. Wachtel \cite{DW2015}. More precisely,  we prove  a version of the Stone local limit theorem \cite{stone1965} for aperiodic random walks in cones and  then  a local central  limit theorem  for these walks.

 To prove  local limit theorem in the one-dimensional case,  i.e. for random walks conditioned to stay positive,  the usual way is to start with  the Wiener-Hopf factorisation \cite{LP1},  \cite{VW}. This powerful tool still exists for half spaces in higher dimension \cite{LP2}. For more general cones,  the question remained open for a few decades,  and needed a totally different approach. In \cite{DW2015},  the two authors proposed  such a new strategy,  based on a strong approximation of multidimensional random walks with multidimensional Brownian motion; for Brownian motion,  the study of its  exit times from a cone  and its  local behavior have been the subject of a wide literature. We refer to \cite{DW2015} for a presentation of the subject and references therein. We adapt here D.  Denisov and V. Wachtel approach to the case of nonlattice random walks.

 The paper is organized as follows: In Section \ref{Preliminaries},  we collect some basic notions of random walks in cones and state the main results of the paper concerning a local limit theorem for nonlattice random walks in cones (Theorem \ref{theostonecone}) and the asymptotic behavior of  random walks conditioned staying in cone (Theorem \ref{theolocal}). Section \ref{preparatory} is devoted to establishing several preparatory results. The proofs of Theorem   \ref{theostonecone} and Theorem \ref{theolocal} are given in section  \ref{sectionproofTheo1} and section \ref{sectionproofTheo2},  respectively.

\section{Preliminaries and main results}\label{Preliminaries}

\subsection{Random walks in cones}
 
We endow $\bR^d$ with the canonical scalar product $\langle  \cdot,  \cdot \rangle$  and let  $\{e_1,  \ldots,  e_d\}$  be  the canonical orthonormal basis  of the Euclidean space $\bR^d$.

 Denote by $\bS^{d-1}$ the unit sphere of $\bR^d$ and $\Sigma$  an open and connected subset of $\bS^{d-1}$. Let $\cC:= \{t{\bf x}\mid t>0 \ {\rm and} \ {\bf x} \in \Sigma\}$ be the cone generated by the rays emanating from the origin and passing through $\Sigma$.
 
We consider a random walk   $(S(n))_{n \geq 0}$ on $\bR^d$,  where  $S(0)=0$,  
\[
S(n)= X_1+\cdots+X_n
\]
 and $\{X_n,  n \geq 1\}$ is a family of independent copies  (whose distribution  is denoted $\mu$) of a random $d$-dimensional vector $X= (X^1,  X^2,  \ldots,  X^d)$ defined on a probability space $(\Omega,  \mathcal F,  \bP)$. For any $n \geq 1$,   we denote by $\mathcal F_n$ the $\sigma$-algebra generated by the random variables $X_1,  \ldots,  X_n$; by convention $\cF_0= \{\emptyset,  \Omega\}$. For any ${\bf x} \in \bR^d$,  the family  $({\bf x} +S(n))_{n \geq 0}$ is the random walk   on $\bR^d$   starting from    $\bf x$ and with steps $X_i$. 
 
 Throughout this paper,  we assume  ${\bf x} \in \cC$. In order to control the behavior of the random walk  $({\bf x} +S(n))_{n \geq 0}$   before leaving the cone $\cC$,   we introduce the following random variable $\tau({\bf x})$,  called ``{\it  the first exit time  of $\cC$}'',  
\[
\tau({\bf x}):= \inf\{n\geq 1\mid{\bf x}  +S(n) \notin \cC\}.
\]
The  random variables $\tau({\bf x})$ are stopping times with respect to the filtration $(\cF_n)_{n \geq 0}$. When the random vectors $X_i$ have moment of order 1 and are centered,  the  random variables $\tau({\bf x})$ are finite $\bP$-a.s.; equivalently,  the sequence $(\bP(\tau({\bf x})>n))_{n \ge 0}$   tends to $0$ as $n \to +\infty$. Under some stronger moment conditions,  D. Denisov and V. Wachtel   proved in   \cite{DW2015} that this  sequence behaves as $n^{-p/2}$ for some explicit  constant $p>0$ (see below).  In the case where the random walk  $(S(n))_{n \geq 0}$ takes values in  $\mathbb Z^d$,   they also control    the asymptotic behavior of the sequences $(\mathbb P({\bf x} +S(n)={\bf y},   \tau({\bf x})>n))_{n \geq 0}$  for  ${\bf x},  {\bf y}  \in  \mathbb Z^d\cap \mathcal C$. The main goal of the present section  is  to extend their result  for a general $\bR^d$-valued random walk.

  D. Denisov \& V. Wachtel's approach is based on the universality property of the Brownian motion. Let $B(t)$ be a standard Brownian motion  on $\bR^d$ and let   $\tau^{bm}({\bf x})$ be the exit time of $B(t)$ from the cone $\cC$, 
  \[
\tau^{bm}({\bf x}):= \inf\{t\geq 0\mid{\bf x}  +B(t) \notin \cC\}.
\]
 Let $\Delta$ be the   Laplace-Beltrami  operator on $\mathbb R^d$. An important role is played by the harmonic function  $u$ of $B(t)$ killed at the boundary of $\mathcal C$; this function $u$ is  the unique (up to a constant) minimal   solution  of the  boundary problem:
\[
\Delta u({\bf x})= 0 \quad   {\rm for} \quad {\bf x} \in \cC
\]
 where   $ u$ is positive   on  $ \mathcal C$ and satisfies  the   boundary  conditions    $u  \vert_{\partial \mathcal C}=0$.

The function $u$  can be found as follows.   Let $\Delta_{\mathbb S^{d-1}}$ be the Laplace-Beltrami operator on  $\mathbb S^{d-1}$. There exists a complete set of orthonormal eigenfunctions $m_j$ and corresponding eigenvalues $\lambda_j,  j \geq 1, $ satisfying
\begin{equation}\label{laplace}
 \left\{
 \begin{array}{cl}
\ 0<\ \lambda_1 \  <\ \lambda_2& \leq\ \lambda_3 \  \leq \ldots  
\\
\Delta_{\mathbb S^{d-1}}m_j({\bf x})&= \ -\lambda_jm_j({\bf x}),  \quad {\bf x} \in \Sigma \\
m_j({\bf x})&=\ 0,  \quad {\bf x} \in \partial \Sigma. 
\end{array}
\right.
\end{equation}
Let $p= \sqrt{\lambda_1+(d/2-1)^2}-(d/2-1).$ It holds $p>0$ and the function $u$ is given by: for any ${\bf x} \in \cC,  {\bf x}\neq {\bf 0}$,   
\begin{equation}\label{homogene-u}
u({\bf x})= \vert {\bf x}\vert^p m_1\left({{\bf x}\over \vert {\bf x}\vert}\right).
\end{equation}
The existence of  a  harmonic function  $V$ for the random walk $({\bf x} +S(n))_{n \geq 0}$  being killed at $\tau({\bf x})$ is  much more delicate to  prove,  we refer to \cite{DW2019} for  details  and comments. Some restrictive assumptions on $\cC$ are required,  depending in particular of the parameter $p$ defined above.  

We  introduce the following assumptions.

\noindent {$\bullet $    \underline{Cone assumption   {\bf C}}: {\it the cone $\cC$ is either starlike with $\Sigma $ in $C^2$ or convex.}

Recall   that  $\cC$ is starlike if there exists ${\bf a}_0 \in \Sigma$ such that ${\bf a}_0+\cC\subset \cC$ and dist$({\bf a}_0+\cC,  \partial \cC)>0$. In particular,  every convex cone is also starlike.  Notice that  in this case,  the set  of points ${\bf a}\in \mathcal C$ such that ${\bf a} +\cC\subset \cC$ has non empty interior.

We impose the following assumptions on the increments  $X_n$ of the random walk.

 \noindent  $\bullet $  \underline{Moment assumption  
   {\bf M}}:  
{\it We assume that $\mathbb E(\vert X \vert^\alpha)<+\infty$ with $\alpha = p$ if $p>2$ and  $\alpha >2$ when $p\leq 2$. }
 
  \noindent  $\bullet $  \underline{Normalization assumption {\bf N}}: {\it  We assume that  $\bE(X)={\bf 0}$,   $\bE((X^i)^2)=1$ and cov$(X^i,  X^j)= 0$ for $1\leq i,  j\leq d$ with $ i\neq j$.}
 
 Depending on the distribution $\mu$  of the $X_i$, the  random walk starting from ${\bf x} \in \mathcal C$ and conditioned to stay in $\mathcal C$ up to a certain time $n$ may non visit the whole cone $\mathcal C$.  Following \cite{DW2015}, we thus introduce the subset $\mathcal C_\mu$ of $\mathcal C$ defined by
\[
\mathcal C_\mu := \bigcup_{\gamma >0}\bigcap_{R>0}\bigcup_{n \geq 1}\{{\bf x}\in  \mathcal C\mid  \mathbb P({\bf x}+S(n) \in D_{ \gamma, R}, \tau({\bf x}) >n)>0  \} 
\]
where $D_{ \gamma, R}:= \{ {\bf x} \in \mathcal C\mid {\rm dist}({\bf x}, \partial \mathcal C)\geq  \gamma \vert {\bf x}\vert \ {\rm and} \ \vert {\bf x}\vert \geq R \}.$ The set $\mathcal C_\mu$ is not a cone in general and may be a proper subset of $\mathcal C$; let us give an example to illustrate this phenomenon.
\begin{example}
We consider the case $d=2$ with the cone $\mathcal C:= \mathbb R^{*+}\times \mathbb R^{*+}$.

Let the distribution $\mu$ of the jumps $X_i$  be  the centered probability measure  $\mu:= {1\over 4} \delta_{(2, -1)}+{1\over 4}\delta_{(0, -1)}+{1\over 2}\delta_{ (-1, 1)}$ where $\delta_{\bf x}$ denotes the Dirac measure at ${\bf x}$.  The  semi-group generated by the support of $\mu$   equals $\mathbb Z^2$ and for any ${\bf x} \in \mathbb R^2$, the random walk $({\bf x} +S(n))_{n \geq 0}$ may visit   any site  in $\mathbb Z^2$ with a positive probability. 

Nevertheless, for ${\bf x} \in   \mathcal C$, when we condition the path   ${\bf x} +S(n)$ to stay  up to time $n$ in $\mathcal C$, the choice of the starting point ${\bf x}$ is more restrictive.  
 
- If ${\bf x}\in ]0, 1]\times ]0, 1]$ then  $\mathbb P({\bf x}+S(1) \in \mathcal C)=0$, hence ${\bf x} \notin \mathcal C_\mu$. 

- If    $x_1>1$   then ${\bf y}= {\bf x}+(-1, 1) \in \mathcal C$ and   dist$({\bf y}, \partial \mathcal C)=\min(x_1-1, x_2+1)$; consequently, it holds  
\[ \mathbb P({\bf x}+S(1)\in D_{\gamma, R}, \tau({\bf x})>1)\geq \mu\{(-1, 1)\}=1/4
\]
for $\gamma= \gamma_{\bf x}= {\min(x_1-1, x_2+1) \over \sqrt{(x_1-1)^2+(x_2+1)}}$ and any $R>0$.
By definition of $\mathcal C_\mu$, this proves that ${\bf x}\in \mathcal C_\mu$ when $x_1>1$.
 
 - If $0<x_1\leq  1$ and $x_2> 1$, the same argument holds with  ${\bf y}= {\bf x}+(0, -1)$.\ 
Finally, $ \mathcal C_\mu=\mathcal C\setminus  (]0, 1]\times ]0, 1])$.\
\end{example}

We now quote  the results from \cite{DW2015} and  \cite{DW2019} about a construction and some properties of  a harmonic function $V$  for the random walk $\{{\bf x} +S(n)\}_{n \geq 0}$ killed at $\cC$ in terms of the function $u$.   We emphasize that in both papers the construction of $V$ is valid for $\mathbb R^d$-valued random walks; in  \cite{DW2019}, the hypotheses on the cone $\mathcal C$ are less restrictive than in \cite{DW2015} and it is this version that we quote here.

\begin{prop}\cite[Lemmas 3.6]{DW2019} {\rm Harmonic function for random walks killed at the boundary of $\cC$.} \label{HarF}
Assume hypotheses {\bf C},  {\bf M} and {\bf N}. Then,  the function $V : \mathcal C \rightarrow \bR^+$ defined by
\[
V({\bf x}) := \lim_{n \to \infty} {\bE}(u({\bf x} +S(n)); \tau({\bf x})>n)
\]
is well-defined,  finite and harmonic for $({\bf x} +S(n))_{n \geq 0}$ killed at leaving $\mathcal C$; in other words, 
\[
V({\bf x})= \bE(V({\bf x} +S(n)); \tau({\bf x})>n),  \quad {\bf x} \in \cC,  n \geq 1.
\]
Furthermore,  the function $V$ satisfies the following properties:
 	\begin{itemize}
	 		\item[(i)] (Positivity) The function 	$V$ is positive    on the set  $\mathcal C$.
			
		\item[(ii)] (``monotony" inside the cone)   If ${\bf x} \in \cC$ then $V({\bf x})\leq V({\bf x} +{\bf a})$ for  all  ${\bf a}\in \mathcal C$ such that ${\bf a} +\cC\subset \cC$.
		
		\item [(iii)]There exists $C_V>0$ such that $0\leq V({\bf x}) \leq C_V(1+\vert {\bf x}\vert^p)$ for any ${\bf x} \in \cC$.
				
		\item[(iv)] (Asymptotic behavior)  For any ${\bf x} \in   \Sigma, $
		\[V(t{\bf x})\sim t^p m_1({\bf x})\quad {\rm  as } \quad t \to \infty, \]
		the convergence being  uniform on compact set of $ \Sigma$.
		\end{itemize}
 	\end{prop}

 The function $V$ contains   informations on the distribution of the exit time $\tau_{\bf x}$  from the cone $\cC$  of the random walk $({\bf x}+S(n))_{n \geq 0}$. More precisely, it holds:

\begin{prop}\label{exitTime} \cite[Theorems 1 \& 3]{DW2015}  Under hypotheses {\bf C},  {\bf M} and {\bf N},  there exists  a positive constant  $\mathrm{K}$ such that,  for any ${\bf x} \in \cC$, 
	\[
	\mathbb P (\tau({\bf x}) >n) \leq K  {1+ \vert x\vert^p\over n^{p/2}}.
	\]
	Furthermore,  
 there exists an universal  constant $\kappa_0>0$  such that
 	\[
 	\mathbb P  (\tau({\bf x}) >n) \sim \kappa_0 V({\bf x})n^{-p/2}\quad\mbox{as} \quad\,  n \to +\infty 
 	\]
 and   
	\[ \bP\Bigl({{\bf x} +S(n)\over \sqrt{n}}\in  \cdot  \Big\vert \tau({\bf x}) >n\Bigr)\longrightarrow \mu(\cdot) \quad  weakly, 
	\]
	where  $\mu$ is the probability measure on $\cC$ having density $  H_0 u({\bf y})   e^{-\vert{\bf y}\vert^2/2} $  with respect to the Lebesgue measure with    $H_0= \displaystyle\left( \int_{\mathcal C}u({\bf y})  \ e^{-\vert{\bf y}\vert^2/2} {\rm d}{\bf y}\right)^{-1}.$ 
	
\end{prop}

 The universality of the constant $\kappa_0$ above comes from the analogous result valid for the Brownian motion \cite{BS}. Namely,  for any point  $\bf x \in \mathcal C$, if $\tau^{bm}({\bf x})$ denotes the first exit time from $\mathcal C$ of the Brownian motion starting from  $\bf x $, it holds  
 $$
 \mathbb P(\tau^{bm}({\bf x})>t)\sim \kappa_0 {u({\bf x})\over t^{p/2}}\quad {\rm as } \ t \to +\infty.
 $$

 \subsection{Main results}

Before stating the main results of this paper,   we make precise the definition of lattice and nonlattice distribution  in the multidimensional context. 
 We refer the reader to    \cite[Chap.VII, $\S$1, Theorem 2]{Bour}.   

Let $G$ be the closed subgroup of $\mathbb R^d$ generated by the support $S_\mu$ of $\mu$ and $V$  the biggest vectorial subspace of $\mathbb R^d$ which is included in $G $. Set $s:= {\rm dim} V$ and $G_1= G \cap V$. There  exists a discrete subgroup  $G_2$ of $G $ such that $G = G_1\oplus G_2$; the subgroup $G_2$ has rank $t= d-s$, in other words, there exists linearly independent elements $u_1, \ldots, u_{t}$ in $G$ such that $G_2= \mathbb Z u_1\oplus \cdots \oplus \mathbb Z u_{t}$.

The  structure  of $G$ induces some strong properties for the characteristic function $\hat{\mu}$ of $\mu$, defined by:
\[
\forall \theta   \in \mathbb R^d, \quad \hat{\mu}(\theta):= \mathbb E[e^{i\langle \theta, X_1\rangle}].
\]

Any element ${\bf x}\in G$ may be decomposed in a unique way as ${\bf x}= {\bf x}_1+{\bf x_2}$ with ${\bf x}_1\in G_1$ and  ${\bf x}_2\in G_2$;   we also write in the sequel ${\bf x}= ({\bf x}_1, {\bf x_2})$. The projection $\mu_1$ of the measure $\mu$ on $G_1$ is a {\it nonlattice distribution}  on $V$: in other words, its characteristic function ${\hat \mu}_1$    satisfies
\[\forall \theta \in \mathbb R^s \setminus\{{\bf 0}\} \quad \vert {\hat \mu}_1(\theta)\vert <1.
\] 
In contrast, the projection $\mu_2$  of $\mu$ on $G_2$ is a {\it  lattice distribution}  on $G_2$ whose support $S_{\mu_2}$ generates $G_2$.  Furthermore, its characteristic function ${\hat \mu}_2$ is  periodic on $\mathbb R^d$: indeed,    for any $\theta \in \mathbb R^s$ and $k_1, \ldots, k_t \in \mathbb Z$,
\[ {\hat \mu}_2(\theta)= {\hat \mu}_2(\theta+2\pi k_1 u_1/{\vert u_1\vert^2}+\ldots+2\pi k_t u_t/{\vert u_t\vert^2}).\]
In this case, the inequality  $\vert \hat{\mu}_2(\theta)\vert <1$  holds for any $\theta \in [-\pi/\vert u_1\vert, \pi/\vert u_1\vert]\times \cdots \times [-\pi/\vert u_t\vert, \pi/\vert u_t\vert]$  with $\theta \neq {\bf 0}$  if and only if the support of $\mu_2$ is not included in a proper coset of $\mathbb Z u_1\oplus\ldots \oplus \mathbb Zu_t$, or, equivalently, if and only if the group generated by $S_{\mu_2}-S_{\mu_2}$ is equal to $G_2$.

When $G_1=\{\bf 0\}$, the measure $\mu$ is supported by the discrete group $G_2$ and  is said {\it lattice}. When $G_1= \mathbb R^d$  it is said  {\it nonlattice}. As explained above,   lattice and nonlattice distributions do not exhaust all possibilities unless $d=1$.

In order to obtain a local limit theorem, we  introduce the following    {\it aperiodicity} condition:

\noindent $\bullet$ \underline{Aperiodicity assumption  {\bf A }}:  
{\it The measure $\mu$ is said aperiodic  if   its characteristic function $\hat \mu$  satisfies 
\[\vert {\hat \mu}(\theta)\vert <1\]
 for any 
 $\theta =(\theta_1, \theta_2) \in \bR^s\times \left( [-\pi/\vert u_1\vert, \pi/\vert u_1\vert]\times \cdots \times [-\pi/\vert u_t\vert, \pi/\vert u_t\vert]\right)$ such that $\theta \neq  {\bf 0}$.}

 The following statement is  a version, for multidimensional aperiodic random walks in a cone, of the Stone's local limit theorem   \cite{stone1965} in the non lattice case. This is  the main statement of the present paper. The proof follows the strategy developed in \cite[Theorem 1.5]{DW2015} and \cite[Corollary 1.3]{DW2019}  in the lattice case;  

\begin{theo}[\rm  Stone  local limit theorem for nonlattice random walk in cones]\label{theostonecone}
Assume  hypotheses  {\bf C},  {\bf M},  {\bf N} and {\bf A}. Let ${\bf B}_1\subset \mathbb R^s$ be a  bounded Borel  set  whose boundary is negligible with respect to the Lebesgue measure on $\mathbb R^s$   and $ {\bf y}_2\in G_2$ such that ${\bf B}:= {\bf B}_1\times \{{\bf y}_2\}\subset \mathcal C$.  Then,    as $n \to +\infty$,  uniformly in $ {\bf x} \in \cC$,
 \begin{equation*}\label{gnedenko}
 \Big\vert n^{p/2+d/2}\bP(  \tau({\bf x})>n, {\bf x}  +S(n) \in   {\bf B})- \    \kappa_0 H_0  V({\bf x})\  \int_{\bf B}  u( {\bf y}/\sqrt{n})e^{-{\vert  {\bf y}_1\vert^2+ \vert{\bf y}_2\vert^2\over 2n}} {\rm d}{\bf y}_1\Big\vert\to 0 
\end{equation*}
where $\kappa_0$ and $H_0$ are defined in Proposition \ref{exitTime}.
 \end{theo}
As a consequence,  we obtain the asymptotic behavior of the return probabilities sequence $( \bP(  \tau({\bf x})>n,  {\bf x}  +S(n) \in{\bf B}) )_{n \geq 0}$. Note that   Theorem  \ref{theostonecone} says only that  this probability is $o( n^{p/2+d/2})$.
\begin{theo}[\rm Asymptotic behavior of the return probabilities for random walks  in cone]\label{theolocal}

Assume  hypotheses  {\bf C},  {\bf M},  {\bf N} and {\bf A}.  Let ${\bf B}_1\subset \mathbb R^s$ be a  bounded Borel  set  whose boundary is negligible with respect to the Lebesgue measure on $\mathbb R^s$   and $ {\bf y}_2\in G_2$ such that ${\bf B}:= {\bf B}_1\times \{{\bf y}_2\}\subset \mathcal C$.  Then,  there exists  a  constant $\mathrm{C}$ such that,    for any $n \geq 1$ and $  {\bf x} \in \cC$,
\begin{equation}\label{localmaj}
\bP( \tau({\bf x})>n, {\bf x}  +S(n) \in  {\bf B})\leq   
\mathrm{C} \  { 1+\vert {\bf x}\vert^p \over n^{p+d/2}} \  \int_{{\bf B}_1}\widetilde{V}({\bf y}_1, {\bf y}_2) {\rm d} {\bf y}_1.
\end{equation}
Furthermore,   as $n \to +\infty$,  
\begin{equation}\label{local}
\bP( \tau({\bf x})>n,  {\bf x}  +S(n) \in  {\bf B})\sim  
\mathfrak \kappa_1  { V({\bf x})\over n^{p+d/2}} \int_{{\bf B}_1}\widetilde{V}({\bf y}_1, {\bf y}_2) {\rm d} {\bf y}_1,
\end{equation}
where $\widetilde V$ is the harmonic function  for the random walk $(-S(n))_{n \geq 0}$   killed at   time $\tilde{\tau}$ of leaving $\mathcal C$ and $\kappa_1= \kappa_0^2H_0^2  \int_{\mathcal C} u^2({\bf y}) e^{-\vert {\bf y}\vert^2/2}{\rm d}{\bf y}.$  
\end{theo}

 The lattice case  $G_1=\{\bf 0\}$  has been studied in \cite{DW2015}.   
 The fact that random walk$(S(n))_{n \geq 0}$ is non lattice here  introduces additional technical difficulties. Indeed, for all points ${\bf x}, {\bf y} \in \mathcal C$ the probability of events $({\bf x}+S(n)={\bf y})$ is zero, so the arrival point ${\bf y }$must be replaced by a Borel set of non-zero Lebesgue measure containing it. This operation greatly complicates the time reversal process,  which is the key point of the proof. We overcome  in Section \ref{Subsection3.2} 
 this difficulty by thickening and reducing the cone $\mathcal C$.

 In the following sections, in order to simplify the notations, we focus on the nonlattice case   $G_1=\mathbb R^d$. The mixed case, neither lattice nor non lattice,  is an easy  consequence of these two cases, by using the decomposition $G= G_1\oplus G_2$.  From now on, we thus assume $G_1= \mathbb R^d$ and set ${\bf y}_1={\bf y} $ and ${\bf y}_2=0$ in formulae \eqref{localmaj} and \eqref{local}.
 
  To give  meaning to the statement of Theorem \ref{theolocal} in the nonlattice case,  it is important to notice that  $ V $ and $\widetilde{V}$ are locally Riemann integrable. This is a consequence of  monotony  property of these functions inside the cone  stated in Proposition \ref{HarF},   we provide  detailed   arguments in 
  Subsection \ref{Subsection3.1}. 

  \begin{notation}  For any sequences $u=(u_n)_{n \geq 0}, v=(v_n)_{n \geq 0}$ of  positive functions, we write $u\preceq v$, or $u_n\preceq v_n$ if there exists a constant $c>0$ such that $u_k \leq cv_k$ for any $k \geq 1.$
\end{notation}

 \section{Preparatory results}\label{preparatory}
 
 This section is devoted to establishing several preparatory results for the proof of the main results. These results are:
\begin{itemize}
\item the local Riemann integrability of the harmonic function $V$ for random walk killed at the boundary of $\cC$ (Subsection \ref{Subsection3.1});
\item the approximation of the reverse exit time (Subsection \ref{Subsection3.2});
\item   the estimate on the probability $\bP(\tau({\bf x}) >n,  {\bf x}  +S(n) \in {\bf y}+ \mathrm{P})$,  where ${\bf x}, {\bf y}$ belong to   $\cC$ and $\mathrm{P}$  is a  neighborhood of the origin ${\bf 0}$ which is a product of intervals (Subsection \ref{Subsection3.3}).
\end{itemize}

\subsection{Local Riemann integrability of the harmonic function $V$}\label{Subsection3.1}

Recall from the content of Proposition \ref{HarF}  that the harmonic function $V:\mathcal C\rightarrow \mathbb R^+$ for the random walk $S(n)$ killed at the boundary of $\mathcal C$ is defined by 
\begin{equation}\label{Harmonicfunction}
V({\bf x}):=\lim_{n\to\infty}  \bE(u({\bf x} +S(n)); \tau({\bf x})>n).
\end{equation}
  It is well known that the function $u$ is analytic on $\mathcal C$.  In some very special cases in the choice of $\mathcal C$ and $\mu$,  the same property holds for $V$ since it is a  polynomial. See for instance \cite{HR} in the case where $\mathcal C$ is the  Weyl chamber  of some finite Coxeter groups  and $\mu$ a probability distribution ``adapted"  to $\mathcal C$. See also \cite{DEW}  for general cones  but with  some restrictive assumptions on the distribution $\mu$. Nevertheless, in general, the function $V$ is not even   continuous.  We illustrate this fact with examples.
\begin{examples}
In dimension 1, when $\mathcal C$ is the half-line $(0, +\infty)$, it holds, for any $x \geq 0$, 
\[
V(x) = V_\mu(x)= \sum_{n \geq 0} \mu_+^{\star n}([0, x]),
\]
where $ \mu_+$ is the distribution of the first strict ascending ladder height  of the random walk  $(S(n))_{n \geq 0}$ (see for instance \cite[Theorem 2.6]{LP2}).
In some cases,  the function $V_\mu$ may be non continuous on $\mathbb R^+$. 

\begin{enumerate}
\item   When $\mu$ is supported by $\mathbb Z$, $V_\mu$ is a step function with jumps in $\mathbb N_0$.

\item  If $\mu:= {1\over \sqrt{2}+1} \delta_{\sqrt{2}}+{\sqrt{2}\over \sqrt{2}+1} \delta_{-1}$ then  the measure $\mu_+$ is supported by the set $\{ p\sqrt{2}-q\mid p, q \in \mathbb N\}\cap ]0, \sqrt{2}]\}$, hence   $\displaystyle \sum_{n \geq 0} \mu_+^{\star n}$ is supported by   $T_+(\sqrt{2}, -1):= \{ p\sqrt{2}-q\mid p, q \in \mathbb N\}\cap \mathbb R^{*+}$. The set   $T_+(\sqrt{2}, -1)$ is  dense in $\mathbb R^{*+}$ and it coincides with  the set of discontinuity of  $V_\mu$ since $V_\mu(x)-V_\mu(x+0^-)=\mu_+(x)$ for any $x \in \mathbb R^*+$.  

\item In higher dimension, the same phenomenom may appear. Indeed, if $\mathcal C = \mathbb R^{*+}\times \mathbb R^{*+}$ and $\mu=\nu_1\otimes \nu_2$ where $\nu_1$ and $\nu_2$  are both centered probability measure on $\mathbb R$  then   $V_\mu=  V_{\nu_1} \otimes V_{\nu_2}$,   where $  V_{\nu_i}, i= 1, 2$,  is the harmonic function on $\mathbb R$ corresponding to $\nu_i$ and  the half-line $(0, +\infty)$. This can be seen for instance as a consequence of \cite[Theorem 2.6]{LP2} since the coordinates of  the random walk with jump distribution $\mu$ on $\mathbb R^2$ are independant. Consequently, if      $\nu_1=\nu_2= {1\over \sqrt{2}+1} \delta_{\sqrt{2}}+{\sqrt{2}\over \sqrt{2}+1} \delta_{-1}$ as in $2)$ then  the harmonic function   is discontinuous  set of $V_\mu$ in $\mathcal C$ equals    $ T_+(\sqrt{2}, -1)\times \mathbb R^{*+}\cup  \mathbb R^{*+}\times T_+(\sqrt{2}, -1).$ 
\end{enumerate}  
\end{examples}
Nevertheless, in dimension 1,    $V$ is monotone  hence Riemann integrable.  
In higher dimension, the  ``monotony" of $V$ inside the cone, given by Proposition \ref{HarF}, also yields its local Riemann integrability.  Let us explain this. 

By a classical Riemann's argument,  the local Riemann integrability of $V$ is equivalent to the fact that  the set of discontinuities of $V$ is Lebesgue negligible.
The proof we present here  is due to E. Lesigne \cite{Lesigne2021}  and is decomposed into 2 steps.

 Let ${\bf x}\in \cC$ be a point of discontinuity of $V$ and ${\bf A}$ the subset of $\mathcal C$ defined by 
\begin{equation}\label{SetA}
{\bf A}:=\{{\bf a}\in\mathcal C: {\bf a}+ \mathcal C \subset \mathcal C, \; \mathrm{dist}({\bf a}+\cC,  \partial \cC)>0\}.
\end{equation}

\noindent  \underline{Step1.} {\it There exists ${\bf a}\in {\bf A}$ such that $\bf x$   is  a  point of  discontinuity of  $V$ in the direction ${\bf a}$. }

By Assumption {\bf C},  the set ${\bf A}$ is open and not empty and by Proposition \ref{HarF} (ii) the function $V$ is monotonically increasing inside the cone with respect to the translation associated with a vector ${\bf a}\in {\bf A}$.  
 Hence,  there exist two sequences $({\bf u}_n)_{n \geq 0}$ and $({\bf v}_n)_{n \geq 0}$  which converge towards ${\bf x}$ and constants $l<L$ such that $V({\bf u}_n)\leq l<L\leq V({\bf v}_n)$ for any $n \geq 0$.
Since the set ${\bf A}$ is open,  there exist    vectors  ${\bf n}\neq 0$  and ${\bf b,  c  \in  A}$ such that  
the hyperplane $P:=(\bR {\bf n})^\perp  $ intersects ${\bf A}$ and   satisfy $\langle {\bf b},  {\bf n}\rangle >0$ and $\langle {\bf c},  {\bf n}\rangle <0$.
Without loss of generality,  we may assume  that,  for some open neighborhood  $U$ of ${\bf x}$,   all  the points ${\bf u}_n,  {\bf v}_n$ and their projections on the affine hyperplane ${\bf x}+P$  along the directions $ \bR{\bf b}$ and $ \bR{\bf c}$  belong to $V$. 

Now,  for each point ${\bf u} \in U$,  let us denote ${\bf u}+ \beta {\bf b}$ (resp. ${\bf u}+ \gamma {\bf c}$) its projection on ${\bf x} + P$ in the direction  ${\bf b}$ (resp. ${\bf c}$). By the hypothesis on the respective positions of ${\bf b}$ and ${\bf c}$ with respect to $P$,  it holds $\beta \gamma\leq 0$ with equality if and only if $\beta=\gamma=0$. We set 

$\bullet$ ${\bf u}^+={\bf u}+ \beta {\bf b}$ and 
${\bf u}^-={\bf u}+ \gamma {\bf c}$ when $\beta> 0$ and $\gamma< 0$;

$\bullet$    ${\bf u}^+={\bf u}+ \gamma {\bf c}$ and 
${\bf u}^-={\bf u}+ \beta {\bf b}$ when $\beta< 0$ and $\gamma> 0$.

\noindent By convention ${\bf u} ={\bf u}^+={\bf u}^-$ when $ \beta=\gamma=0$ (i.e. when ${\bf u} \in P)$.

 Then,  for all $n \geq 0, $ the points  ${\bf u}_n^-$ and ${\bf v}^+_n$ belong to $({\bf x} + P)\ \cap \ U$ and,  by monotony property,  they  satisfy
 \[
 V({\bf u}^-_n)\leq V({\bf u}_n)\leq l<L\leq V({\bf v}_n)\leq V({\bf v}^+_n).
 \]
Hence,  we replace $P$ by one of its hyperplanes $Q$ such that   $Q\ \cap \ {\bf A}\neq \emptyset$ and ${\bf b},  {\bf c}$ by vectors in $P\ \cap \ {\bf A}$ who do  not belong to the same half space of $P$ bounded by $Q$. After a   backward recursion with at most $(d-1)$-steps,   we prove the existence of some ${\bf a} \in {\bf A}$,  which does not depends on ${\bf x}$,  and sequences   $({\bf u}_n')_{n \geq 0}$ and $({\bf v}'_n)_{n \geq 0}$  in $\cC$  which converge to ${\bf x}$,  belong to the affine line ${\bf x}+\bR{\bf a}$   and satisfy  $
 V({\bf u}'_n) \leq l<L \leq V({\bf v}'_n)
 $ for any $n \geq 0$.
In other words,  ${\bf x}$ is  a  point of  discontinuity of  $V$ in the direction ${\bf a}$.

\noindent  \underline{Step 2.} \it The set $D_{\bf a}$ of points ${\bf x} \in \cC$ such that   $V$ is discontinuous at ${\bf x}$ in the direction ${\bf a}$ is negligible  with respect to the Lebesgue measure. }
 
Notice that  for any ${\bf x}\in \mathcal C$  the function $t\mapsto V({\bf x}+t{\bf a})$ is defined on a neighborhood of $0$ and that it is non decreasing on the half line  $({\bf x}+\bR{\bf a})\ \cap \ \ \cC$; hence, its discontinuities set is countable,  hence negligible with respect to the Lebesgue measure.

Fubini's theorem readily implies that $D_{\bf a}$ is negligible: indeed,  for any  open subset  $ V$ of the hyperplane $(\bR {\bf a})^\perp$ and  any $\varepsilon>0$ such that  
\[
V\times [\varepsilon,  \varepsilon]:= \{{\bf v}+t{\bf a}  \mid {\bf v} \in V,  t\in [-\varepsilon,  \varepsilon]\}\subset \cC, 
\]
it holds
\[
\int_{V\times [-\varepsilon,  \varepsilon]} 1_{D_{\bf a}}({\bf x}){\rm d} {\bf x} 
=\int_{V}\underbrace{\left(\int_{-\varepsilon}^{\varepsilon}1_{D_{\bf a}}({\bf v}+t{\bf a}){\rm d} t \right)}_{=0 }{\rm d} {\bf v} =0.
\]

\rightline{$\Box$}

 \subsection{Reversing time and key approximation lemma}\label{Subsection3.2}

The main idea in the  proof by D. Denisov and V. Wachtel  of  the local limit theorem for lattice random walks conditioned to stay in the cone $\cC$  is the  reversion of time.  Following \cite{DW2015}, we decompose  the $n$-steps trajectories  of the random walk  between sites ${\bf x}$ and ${\bf y}$  in $\cC$   into two parts. The first one starting from ${\bf x} $  remains  inside $\cC$ up to time $[n/2]$ while the second part of the trajectory  between time $[n/2]$ and time $n$ has the same distribution as the random walk (${\bf y}-S(n))_{0\leq n\leq [n/2]}$ conditioned to stay inside the cone  up to time $[n/2]$. 

This property relies on the notion of reversion of time that we now define by introducing the random time $\tilde{\tau}:\cC\rightarrow \mathbb N$   defined by
\[
\tilde \tau({\bf y}):= \inf\{n\geq 1\mid{\bf y}  -S(n) \notin \cC\}.
\]
 In the nonlattice case,  it is not relevant to fix exactly the position of the arrival  site ${\bf y}$; indeed,  trajectories of the random walk with a fixed arrival site at time $n$ belong to some set of measure $0$. This difficulty  can be overcome by taking a thicker arrival set. We thus introduce the following notations.

\begin{notation}  
For any  ${\bf x},  {\bf y} \in \bR^d$,  we set  $x^i:= \langle {\bf x},  e_i\rangle,  y^i:= \langle  {\bf y},  e_i\rangle$ for  $1\leq i\leq d$.
When $x^i<y^i$ for any $1\leq i\leq d$,  we set
\[
\llbracket {\bf x},  {\bf y} \rrbracket= [x^1,  y^1]\times \ldots \times [x^d,  y^d]
\quad   and  \quad \llbracket {\bf x},  {\bf y} \llbracket= [x^1,  y^1[\times \ldots \times [x^d,  y^d[.\]

\end{notation}
 By assumption {\bf C},  the set $\bf A $ of points ${\bf a}$ such that ${\bf a} +\cC\subset \cC$ has non empty interior. Note that all assumptions  {\bf C},  {\bf M},  {\bf N} and {\bf NL} and the conclusions of the main results of the paper (Theorem \ref{theostonecone} and Theorem \ref{local}) are preserved under a non singular linear transformation; in other words,  the cone $\cC$ and the random walk $S(n)$ may be  replaced by $T\cC $ and $ T S(n)$,  where $T$ is a  $d\times d$ invertible matrix. Then,  {\bf without loss of generality,  we can assume that
 ${\bf 1}:= e_1+\ldots + e_d$  belongs to $ \bf A $}. This implies in particular that there exists $\delta^*>0$ such that 
 \begin{equation}\label{Interialcone}
 \forall{\bf x}\in 
\llbracket (1-\delta^*){\bf 1}, (1+\delta^*){\bf 1} 
\rrbracket,  \quad {\bf x}+\cC\subset \cC.
 \end{equation}
For any $\delta>0$,   let us consider the cones $\mathcal C_{-\delta}$ and $\mathcal C_\delta$ such that 
$\mathcal C_{-\delta}\subset \mathcal C \subset \mathcal C_{\delta}$ where
\begin{itemize}
\item  $\cC_{-\delta}$ is the smallest set of the form $-t{\bf 1}+\cC $ containing  all the  boxes $\llbracket {\bf x},  {\bf x}+\delta{\bf 1}\rrbracket$ which intersect $\cC$;

\item   $\cC_{\delta}$ is the biggest set of the form $t{\bf 1}+\cC$ such that   all the  boxes $\llbracket {\bf x},  {\bf x}+\delta{\bf 1}\rrbracket$ which intersect  $t{\bf 1}+\cC$ are included in $\cC$.
\end{itemize}
\begin{figure}[h]
\centering
\setlength{\unitlength}{1mm}

\begin{picture}(120,65)(-35,-30)


\color{red}
\linethickness{3pt}
\put(-20,0){\line(4,1){90}}   
\put(-20,0){\line(4,-1){90}}  
\put(71,22){\color{red}$\mathcal C_{-\delta}$}

\color{black}
\linethickness{3pt}
\put(0,0){\line(4,1){70}}    
\put(0,0){\line(4,-1){70}}   
\put(71,16){$\mathcal C$}

\color{blue}
\linethickness{3pt}
\put(20,0){\line(4,1){50}}   
\put(20,0){\line(4,-1){50}}  
\put(71,10){\color{blue}$\mathcal C_\delta$}

\color{red}
\linethickness{1pt}
\put(12,2){\framebox(5,5){}}

\color{blue}
\linethickness{1pt}
\put(32,2){\framebox(5,5){}}

\end{picture}

\caption{Cones ${\color{blue} \mathcal C_{\delta}}\subset \mathcal C\subset {\color{red} \mathcal C_{-\delta}}$}
\centerline{Any box {\color{blue} $\Box$} which intersects {\color{blue} $\mathcal C_{\delta}$}  is included in $ \mathcal C$ 
 }
 \centerline{Any box {\color{red} $\Box$} which intersects $\mathcal C$   is included in {\color{red} $\mathcal C_{-\delta}$}  
 }
\end{figure} 
Let us be more precise. Let $I_\delta$ be the subset of $ \mathbb R^+$ defined by 
\[
I_\delta:= \{t>0\mid \llbracket {\bf x},  {\bf x}+\delta{\bf 1}\rrbracket \subset-t   {\bf 1}+\cC\   for\ any \ {\bf x} \in \bR^d \ such\  that \
\llbracket {\bf x},  {\bf x}+\delta{\bf 1}\rrbracket\ \cap \ \cC \neq \emptyset \},
\] 
we set  $t_\delta:= \inf I_\delta $  and, finally,
\begin{equation}\label{ThickerCone_2}
\cC_{-\delta}=-t_\delta    {\bf 1}+\cC\quad \hbox{and}\quad  \cC_{\delta}=t_\delta  {\bf 1}+\cC.
\end{equation}
Notice on the one hand that  $[{\delta\over \delta^*},  +\infty[\subset I_\delta$ (hence $I_\delta\neq \emptyset$) since for $t\geq\delta/\delta^* $,  it holds
\begin{equation} \label{jqcevgrh}
 \llbracket {\bf x}, {\bf x}+\delta{\bf 1}\rrbracket\ \cap \ \cC \neq \emptyset \quad \Longrightarrow \quad \llbracket {\bf x},  {\bf x}+\delta{\bf 1}\rrbracket\subset- t {\bf 1}+\cC.
\end{equation}
Indeed,  let ${\bf x}\in\bR^d$ such that $\llbracket {\bf x}, {\bf x}+\delta{\bf 1}\rrbracket\ \cap \ \cC \neq \emptyset$ and pick an arbitrary ${\bf y}\in \llbracket {\bf x}, {\bf x}+\delta{\bf 1}\rrbracket \ \cap\ \cC$. Then,  any element in $\llbracket {\bf x},  {\bf x}+\delta{\bf 1}\rrbracket$ may be decomposed as  ${\bf y}+{\bf z}$ where ${\bf z} \in \llbracket-\delta {\bf 1},  \delta {\bf 1}\rrbracket$. Hence, 
\[
 t {\bf 1}+ {\bf y}+{\bf z} = t \left({\bf 1}+ \frac{\bf y}{t }+\frac{\bf z}{t }\right)
\  {\rm with} \ 
  \frac{\bf y}{t } \in \mathcal C\  {\rm and}\ 
{\bf 1}+\frac{\bf z}{t } \in \llbracket (1- {\delta\over t} ){\bf 1},  (1+ {\delta\over t} ){\bf 1}\rrbracket.
\] 
The condition  $t\geq\delta/\delta^*$ yields  $\llbracket(1- { \delta\over t} ){\bf 1},  (1+ {\delta\over t} ){\bf 1}\rrbracket\subset 
\llbracket (1- \delta^*){\bf 1},  (1+ \delta^*){\bf 1}\rrbracket.
$ 
 This together with \eqref{Interialcone} implies that ${\bf 1}+\frac{\bf y}{t}+\frac{\bf z}{t}\in\cC$,  hence $t{\bf 1}+{\bf y}+{\bf z}\in\cC$. This proves \eqref{jqcevgrh}.
 
 On the other hand, $t_\delta \geq \delta/2$. Indeed, since the box $\llbracket-{\delta\over 2} {\bf 1},  {\delta\over 2} {\bf 1}\rrbracket\cap \mathcal C\neq \emptyset$, the condition  $t \in I_\delta$ yields  $t+\llbracket-{\delta\over 2} {\bf 1},  {\delta \over 2} {\bf 1}\rrbracket\subset \mathcal C$  thus $t > \delta/2$.
   
For short,  the exit times of the sets $\cC_{\delta}$ (respectively $\cC_{-\delta}$)  of  the processes $({\bf x}+S(n))_{n \geq 0} $ and $({\bf y}-S(n))_{n \geq 0} $ with ${\bf x, y} \in \mathcal C_\delta$  (respectively ${\bf x, y} \in \mathcal C_{-\delta}$)   are denoted $\tau_\delta({\bf x})$ and $\tilde \tau_\delta({\bf y}) $  (respectively $  \tau_{-\delta}({\bf x})$  and  $\tilde \tau_{-\delta}({\bf y}))$. When $\delta=0$,  we omit the index.
Similar statement as Proposition \ref{HarF} and Proposition \ref{exitTime} hold for sets $\cC_{\delta}$ and $\cC_{-\delta}$; the corresponding functions $V_{\delta},  \tilde V_{\delta},  V_{-\delta}$ and $\tilde V_{-\delta}$ are defined by $ V_\delta({\bf x}):= V({\bf x}+t_\delta{\bf 1}),    \tilde V_\delta({\bf y}):= \widetilde V({\bf y}+t_\delta{\bf 1}),   V_{-\delta}({\bf x}):= V({\bf x}-t_\delta{\bf 1})$ and  $\tilde V_{-\delta}({\bf y}):= \tilde V({\bf y}-t_\delta{\bf 1})$ respectively.

The following result establishes an approximation between exit time and reverse exit time for random walks in cone and plays a crucial in the sequel.

\begin{prop} \label{keyinclusionsGENERALversion}
Let  ${\bf x},  {\bf y} \in \mathcal C$ and $\delta,  \tilde \delta>0$.
Then,  for any  ${\bf z} \in \llbracket {\bf x},  {\bf x}+\delta{\bf 1}\rrbracket\;\cap\;\cC$,  it holds
\begin{equation}\label{key1}
\bP\left( \tau({\bf z}) >n,  {\bf z}  +S(n) \in\llbracket {\bf y},  {\bf y}+\widetilde\delta{\bf 1}\rrbracket \right)
\leq
\bP\left( \widetilde{\tau}_{-\widetilde\delta}({\bf y}) >n,  {\bf y}-S(n) \in \llbracket {\bf x}-\tilde \delta{\bf 1}, {\bf x}+\delta{\bf 1}\rrbracket\right).
\end{equation}

Conversely,  when  $ \widetilde\delta >\delta$ and ${\bf y}\in \mathcal C_{\tilde \delta}$,  it holds 
\begin{equation}\label{key2}
\bP\left(\widetilde{\tau}_{\widetilde\delta}({\bf y})>n,  {\bf y}-S(n) \in \llbracket {\bf x}-(\widetilde\delta-\delta){\bf 1}, {\bf x} \rrbracket\right)
\leq
\bP\left(\tau({\bf x}) >n,  {\bf x}  +S(n) \in\llbracket {\by},  {\by}+\widetilde\delta{\bf 1}\rrbracket\right).
\end{equation}
 
 \end{prop}
{\bf Proof.}
Let us prove inequality \eqref{key1}. Let $\omega\in\Omega$ and $\bf z \in \mathcal C$  such that
\[
{\bf z}+ S(1)(\omega)\in\cC, \dots, {\bf z}+ S(n-1)(\omega)\in\cC,  {\bf z}+ S(n)(\omega)\in\cC\cap \llbracket {\bf y},  {\bf y}+\widetilde\delta{\bf 1}\rrbracket.
\]
From the facts that ${\bf z} \in \llbracket {\bf x},  {\bf x}+\delta{\bf 1}\rrbracket$ and $ {\bf z}  +S(n)(\omega) \in\llbracket {\bf y},  {\bf y}+\widetilde\delta{\bf 1}\rrbracket\;\cap\;\cC$ we derive that
\begin{equation}\label{New_Eq1}
{\bf y} -S(n)(\omega) \in  \llbracket {\bf x}-\widetilde\delta{\bf 1}, {\bf x}+\delta{\bf 1}\rrbracket.
\end{equation}
For $1\leq k\leq n$,  the property  $ {\bf z} +S_k(\omega)\in \cC$ may be rewritten as
\[
 {\bf z}  +S(n)(\omega) - (X_{k+1}(\omega) +\cdots+X_n(\omega))\in \cC;\]
which together with  the fact that $ {\bf z}  +S(n)(\omega) \in\llbracket {\bf y},  {\bf y}+\widetilde\delta{\bf 1}\rrbracket$ implies that
\[
\llbracket{\bf y}-(X_{k+1}(\omega)+\ldots +X_n(\omega)),  {\bf y}-(X_{k+1}(\omega)+\ldots +X_n(\omega))+\widetilde\delta{\bf 1}\rrbracket \, \cap\,  \cC\not=\emptyset, 
\]
thus 
$
{\bf y}-(X_{k+1}(\omega)+\ldots +X_n(\omega))\in \cC_{-\widetilde\delta}$ for $1\leq k\leq n$,  by  definition of  $\cC_{-\widetilde\delta}$.  
By combining this property with \eqref{New_Eq1},  we obtain %
\begin{align*}
\Bigl(\tau({\bf x}) >n,  {\bf x}  +S(n) \in\llbracket {\bf y},  {\bf y}+\widetilde\delta{\bf 1}\rrbracket\Bigr)\subset \Bigl({\bf y}-(X_{k+1} +\ldots +X_n ) &\in \cC_{-\widetilde\delta},  1\leq k  \leq n, 
 \\
 &\cap\ {\bf y}-S_{n} \in \llbracket {\bf x}-\widetilde\delta{\bf 1}, {\bf x}+\delta{\bf 1}\rrbracket\Bigr).
\end{align*}
Since the vectors $(X_1,  \ldots,  X_n)$ and $ (X_n,  \ldots,  X_1)$ have the same distribution,  this readily implies
\begin{eqnarray*}
&& \bP\left(\tau({\bf x}) >n,  {\bf x}  +S(n) \in\llbracket {\bf y},  {\bf y}+\widetilde\delta{\bf 1}\rrbracket \right)\\[1ex]
&\leq&
 \bP \left({\bf y}-(X_{k+1} +\ldots +X_n )\in \cC_{-\widetilde\delta}\ {\rm for} \  1\leq k  \leq n,   {\bf y}-S(n) \in \llbracket {\bf x}-\widetilde\delta{\bf 1}, {\bf x}+\delta{\bf 1}\rrbracket\right)\\[1ex]
 &=&
\bP \Bigl( {\bf y}-S_k\in \mathcal C_{-\widetilde\delta}\ {\rm for} \ 1\leq k  \leq d,  {\bf y}-S(n) \in \llbracket {\bf x}-\widetilde\delta{\bf 1}, {\bf x}+\delta{\bf 1}\rrbracket)\Bigr)\\[1ex]
&=&
\bP\left(\widetilde{\tau}_{-\widetilde\delta}({\bf y})>n,  {\bf y} -S(n) \in \llbracket {\bf x}-\widetilde\delta{\bf 1}, {\bf x}+\delta{\bf 1}\rrbracket\right)
\end{eqnarray*}
and \eqref{key1} follows.

\noindent
Similarly,  for $\tilde \delta>\delta$ and ${\bf y}\in \mathcal C_{\tilde \delta}$,  it holds 
\begin{eqnarray*}
&& \bP\Bigl(\tilde{\tau}_{\tilde \delta}({\bf y})>n,   {\bf y}-S(n) \in \llbracket {\bf x}-(\tilde \delta-\delta){\bf 1}, {\bf x} \rrbracket\Bigr)\\[1ex]
&=&
\bP \Bigl({\bf y}-(X_{k+1} +\ldots +X_n )\in \cC_{\widetilde\delta}\ {\rm for} \  0\leq k  \leq n, {\bf y}-S(n) \in \llbracket {\bf x}-(\widetilde\delta-\delta){\bf 1}, {\bf x} \rrbracket\;\Bigr).
\end{eqnarray*}
Let $\omega \in \Bigl({\bf y}-(X_{k+1} +\ldots +X_n )\in \cC_{\widetilde\delta}\ {\rm for} \  0\leq k  \leq n, {\bf y}-S(n) \in \llbracket {\bf x}-(\widetilde\delta-\delta){\bf 1}, {\bf x} \rrbracket\Bigr)$.
From the condition ${\bf y}-S(n)(\omega)\in \llbracket {\bf x}-(\widetilde\delta-\delta){\bf 1}, {\bf x} \rrbracket$ we derive that
\begin{equation}\label{New_Eq2}
 {\bf x} +S(n)(\omega) \in \llbracket {\bf y},  {\bf y}+(\widetilde\delta-\delta){\bf 1}\rrbracket\subset  \llbracket {\bf y},  {\bf y}+\widetilde\delta{\bf 1}\rrbracket.
\end{equation}
Furthermore,   for $0\leq k\leq n$, the conditions    ${\bf y}-S(n)(\omega)\in \llbracket {\bf x}-(\widetilde\delta-\delta){\bf 1}, {\bf x} \rrbracket$ and       
$$
{\bf y}-(X_{k+1}(\omega)+\ldots +X_n(\omega) )= \underbrace{{\bf y}-S(n)(\omega)}_{\in \llbracket {\bf x}-(\widetilde\delta-\delta){\bf 1}, {\bf x} \rrbracket}+S_k(\omega)\in\cC_{\widetilde\delta},  
$$
 combined with  the definition of $\cC_{-\widetilde\delta}$,    yield  
$
\llbracket {\bf x}-(\widetilde\delta-\delta){\bf 1}+S_k(\omega), {\bf x} +S_k(\omega)\rrbracket\subset \mathcal C
$  
and in particular
$ {\bf x} +S_k(\omega) \in \cC$. Hence,  for this choice of $\omega$ we have $\tau({\bf x})>n$ and this together with \eqref{New_Eq2} shows that $\omega\in \Bigl(\tau({\bf x}) >n,  {\bf x}  +S(n) \in\llbracket {\bf y},  {\bf y}+\widetilde\delta{\bf 1}\rrbracket\Bigr)$. Thus,  we obtain the desired result.
 $\Box$

\subsection{Some useful estimates}\label{Subsection3.3}
 We   state here preparatory overestimations on the local behavior of the random walk $(S(n))_{n \geq 0}$, providing  useful uniform bounds, this is a crucial step towards achieving the announced results. 
\begin{lemma} 
For any $\delta >0$ and   ${\bf x}\in \cC$,  there exist  positive  constants $C({\bf x},  \delta)$ such that  for all  $n\geq 1$ and ${\bf y}  \in \cC$, 
\begin{equation}\label{estimate1} 
  \bP(\tau({\bf x}) >n,  {\bf x}  +S(n) \in\llbracket {\bf y},  {\bf y}  +\delta {\bf 1}\rrbracket)
\quad \leq\quad {C({\bf x},  \delta)\over n^{p/2+d/2}}. 
\end{equation} 
There exist positive constants $c=c(\delta)$ and $C=C(\delta) $ such that  for every $n \geq 1,  t>0$,    and  $ {\bf x},  {\bf y}  \in \cC$,  
 
\begin{align}\label{estimate}
 &(a)\  \bP( {\bf x} +S(n)\in\llbracket {\bf y},  {\bf y}  +\delta {\bf 1}\rrbracket) \leq {C \over n^{d/2}},  \notag
 \\
 &(b)\   \bP( {\bf x} +S(n)\in\llbracket {\bf y},  {\bf y}  +\delta {\bf 1}\rrbracket)
 \leq {C \over n^{d/2}} e^{-c t^2} \quad when \quad  \vert {\bf x}-{\bf y}\vert > t \sqrt{n}, 
\\
&(c)\ 
\bP(\tau({\bf x}) \leq n,  {\bf x}  +S(n) \in\llbracket {\bf y},  {\bf y}  +\delta {\bf 1}\rrbracket)
 \leq {C \over n^{d/2}} e^{-c t^2}\quad when \quad {\rm dist}({\bf x},  \partial \mathcal C),  {\rm dist}({\bf y},  \partial \mathcal C)> t \sqrt{n}.\notag
 \end{align}
\end{lemma}
 Notice that   \eqref{estimate1}  is not optimal,  as we can see by comparing    with \eqref{local}; this first estimation is needed  in the proof  of Theorem \ref{theolocal}. Likewise, the  inequality \eqref{estimate}$(c)$  may seem surprising at first glance  since it concerns the local behaviour in the cone of the random walk at time $n$ when the walk is allowed to leave the cone before that time; it is useful in  Step 3 of the proof of  Theorem \ref{theostonecone}, see Section \ref{sectionproofTheo1}. 

\noindent {\bf Proof.} These inequalities correspond   to Lemmas 27,  28 and 29 in \cite{DW2015}. Assertions  (\ref{estimate}$(a)$) and  (\ref{estimate}$(b)$) are easy consequences of the classical local central limit theorem for random walks on  $\bR^d$ (see \cite{stone1965}). 
\\
Assertion (\ref{estimate1}) follows immediately,  combining  (\ref{estimate}$(a)$)  with  \cite{DW2015} where it is proved that,  uniformly in ${\bf x}\in \cC$, 
 \[\bP( \tau{(\bf x )}>n)\preceq {V({\bf x})\over n^ {p/2}}.\]
 \\
  To prove assertion (\ref{estimate}$(c)$),   we note that,  if ${\rm dist}({\bf y},  \partial \mathcal C)> t \sqrt{n}  $  then,  by using the Markov property,  
\begin{align*}
 & \bP(\tau({\bf x}) \leq n/2,  {\bf x}  +S(n) \in\llbracket {\bf y},  {\bf y}  +\delta {\bf 1}\rrbracket)\\
 &\qquad =\sum_{\ell=1}^{\lfloor n/2\rfloor }
  \bP(\tau({\bf x}) =\ell,  {\bf x} + S(n)  \in\llbracket {\bf y},  {\bf y}  +\delta {\bf 1}\rrbracket)
 \\
 &\qquad \leq \sum_{\ell=1}^{\lfloor n/2\rfloor } \int_{\cC^c}
  \bP(\tau({\bf x}) =\ell,  {\bf x}  +S(\ell)\in {\rm d}{\bf z})\bP({\bf z}+S(n-\ell) \in\llbracket {\bf y},  {\bf y}  +\delta {\bf 1}\rrbracket)
  \\
 &\qquad  \leq \max_{n/2\leq k\leq n} \sup_{\vert {\bf y}-{\bf z}\vert >t\sqrt{n}}\bP({\bf z}+S(k)\in \llbracket {\bf y},  {\bf y}  +\delta {\bf 1}\rrbracket)
 \\
 &\qquad \preceq {e^{-ct^2}\over n^{d/2}}, \quad {\rm for \ some \ constant }\ c>0,  \ {\rm by } \ (\ref{estimate}(b)).
 \end{align*}
To deal with the term $\bP(n/2<\tau({\bf x}) \leq n,  {\bf x}  +S(n) \in\llbracket {\bf y},  {\bf y}  +\delta {\bf 1}\rrbracket)
$,   we first apply the same trick as in Proposition \ref{keyinclusionsGENERALversion},   reversing time. Let $\omega \in \bigl(
n/2<\tau({\bf x}) \leq n,  {\bf x}  +S(n) \in\llbracket {\bf y},  {\bf y}  +\delta \rrbracket\bigr)$; it holds ${\bf y}  -S(n)(\omega) \in \llbracket {\bf x}-\delta,  {\bf x}\rrbracket$ and there exists $n/2< k\leq n$ such that ${\bf x}+S(k)(\omega)\notin \mathcal C$,  ie ${\bf x}+S(n)(\omega)-(X_{k+1}(\omega)+\cdots+X_n(\omega))\notin \mathcal C $. Hence
$$
\llbracket {\bf y} -(X_{k+1}(\omega)+\cdots+X_n(\omega)),  {\bf y} -(X_{k+1}(\omega)+\cdots+X_n(\omega))\rrbracket \not\subset \mathcal C.
$$
By definition of $\cC_{\delta}$,  this implies  $\llbracket {\bf y} -(X_{k+1}(\omega)+\cdots+X_n(\omega)),  {\bf y} -(X_{k+1}(\omega)+\cdots+X_n(\omega))\rrbracket \cap  \mathcal C_\delta =\emptyset $ and in particular ${\bf y} -(X_{k+1}(\omega)+\cdots+X_n(\omega))\notin \mathcal C_\delta. $   Consequently 
 \begin{align*}
 & \bP(n/2<\tau({\bf x}) \leq n,  {\bf x}  +S(n) \in\llbracket {\bf y},  {\bf y}  +\delta {\bf 1}\rrbracket)
 \\
 &\qquad \leq 
 \bP(\exists n/2<k\leq n\mid {\bf y} -(X_{k+1} +\cdots+X_n)\notin \mathcal C_\delta,   {\bf x}  +S(n) \in\llbracket {\bf y},  {\bf y}  +\delta {\bf 1}\rrbracket)
 \\
 &\qquad \leq \bP( \tilde \tau_\delta({\bf y}) \leq n/2,  {\bf y}  -S(n) \in\llbracket {\bf x}-\delta,  {\bf x}  \rrbracket)
 \\
 &\qquad \preceq {e^{-ct^2}\over n^{d/2}}, \quad {\rm by\ the \ same \ argument\ as \ above, \ assuming } \ {\rm dist}({\bf x},  \partial \mathcal C)> t \sqrt{n}.  \end{align*}
\rightline{$\Box$}

\section{Proof of Theorem \ref{theostonecone}}\label{sectionproofTheo1}

 As already mentioned in the introduction just befor the statement of Theorem \ref{theostonecone}, we follow here the proof of \cite[Theorem 1.5]{DW2015} using in a decisive manner Proposition \ref{keyinclusionsGENERALversion}.

 By a classical argument in measure theory, we may assume ${\bf B}:=\llbracket {\bf y}, {\bf y}+\delta{\bf 1}\rrbracket $ with ${\bf y} \in \mathcal C$ and $\delta >0$. 
 We adapt the proof of Theorem 5 in \cite{DW2015} and insist on the main differences. We fix two positive constants $A,  \epsilon$ and split the cone into three parts 
\begin{align*}
\cC^{(1)}&:= \{{\bf y}  \in \cC\mid \vert {\bf y}\vert >A\sqrt{n}\}, \\
\cC^{(2)}&:= \{{\bf y}  \in \cC\mid  \vert {\bf y}\vert \leq A\sqrt{n},  {\rm dist}({\bf y},   \partial \cC)\leq 2\epsilon \sqrt{n}\}, \\
\cC^{(3)}&:=\{{\bf y}  \in \cC\mid \vert {\bf y}\vert \leq A\sqrt{n},  {\rm dist}({\bf y},   \partial \cC)>2\epsilon \sqrt{n}\}.  \end{align*}
The proof is decomposed into 3 steps:
 
\noindent Step 1. 
\begin{equation*}
\lim_{A \to +\infty}\limsup_{n \to +\infty}
n^{p/2+d/2}  \sup_{{\bf y}  \in \cC^{(1)}}\bP( \tau({\bf x})>n,  {\bf x}  +S(n) \in  \llbracket {\bf y}, {\bf y}+\delta{\bf 1}\rrbracket)=0.
\end{equation*}
\noindent Step 2. For any $A>0$,  
\begin{equation*}
\lim_{\epsilon \to 0}\limsup_{n \to +\infty}
n^{p/2+d/2}  \sup_{{\bf y}  \in \cC^{(2)}}\bP(\tau({\bf x})>n,  {\bf x}  +S(n) \in  \llbracket {\bf y}, {\bf y}+\delta{\bf 1}\rrbracket)=0.
\end{equation*}
\noindent Step 3. 
\begin{equation*}
 \lim_{\epsilon \to 0}\limsup_{n \to +\infty} \sup_{{\bf y}  \in \cC^{(3)}}  \Big\vert n^{p/2+d/2}\bP(\tau({\bf x})>n,   {\bf x} +S(n)\in  \llbracket {\bf y}, {\bf y}+\delta{\bf 1}\rrbracket)-c\ V({\bf x})\  u({\bf y}/\sqrt{n}) e^{-\vert {\bf y}\vert^2/2n}\ \delta^d\Big\vert  =0.
\end{equation*}
 Notice that the factor $\delta^d$ equals the volume of the box $\llbracket {\bf y}, {\bf y}+\delta{\bf 1}\rrbracket$, 
this links to the  statement  of Theorem \ref{theostonecone} with ${\bf B}= \llbracket {\bf y}, {\bf y}+\delta{\bf 1}\rrbracket$. 

\
 
In steps 1 and 2, we  set $m=\lfloor n/2\rfloor$.

\ 

\noindent \underline{Proof of Step 1.} Let ${\bf x} \in \cC$ and ${\bf y}  \in \cC^{(1)}$. 

We decompose $\bP( \tau({\bf x})>n,   {\bf x} +S(n)\in  \llbracket {\bf y}, {\bf y}+\delta{\bf 1}\rrbracket)$ as $P_1+P_2$ with 
\[
P_1=\bP(\tau({\bf x})>n,  \vert  {\bf x} + S(m)\vert \leq A\sqrt{m},   {\bf x} +S(n)\in  \llbracket {\bf y}, {\bf y}+\delta{\bf 1}\rrbracket)
\]
and
\[
P_2=\bP(\tau({\bf x})>n,  \vert  {\bf x} +S(m)\vert >A\sqrt{m},   {\bf x} +S(n)\in  \llbracket {\bf y}, {\bf y}+\delta{\bf 1}\rrbracket).
\]
By the Markov property,  inequality   (\ref{estimate}($b$)) and   Proposition \ref{exitTime}, 
\begin{align}\label{estimP1}
P_1 &\leq\bP(\tau({\bf x}) > m,  \vert  {\bf x} +S(m)\vert \leq A\sqrt{m},    {\bf x} +S(n) \in  \llbracket {\bf y},  {\bf y}  + \delta {\bf 1}\rrbracket)\notag\\
&\leq\bP(\tau({\bf x}) >m)\sup_{\vert {\bf y}-{\bf z}\vert>A\sqrt{m}}\bP(   {\bf z}+S(n-m)\in \llbracket {\bf y},  {\bf y}  + \delta {\bf 1}\rrbracket)
\notag\\
&\preceq n^{-p/2-d/2}e^{-cA^2/2 }\end{align}
for some $c>0$.   Similarly,  by  Proposition \ref{exitTime}, 
\begin{align}\label{estimP2}
P_2 &\leq \bP(\tau({\bf x})>m,  \vert  {\bf x} +S(m)\vert >A\sqrt{m},   {\bf x} +S(n)\in  \llbracket {\bf y}, {\bf y}+\delta{\bf 1}\rrbracket)\notag
\\
&\leq \bP\left(\vert  {\bf x} +S(m)\vert >A\sqrt{m}\Big\vert \tau({\bf x})>m\right)
\underbrace{\bP(\tau({\bf x}) >m)}_{\preceq n^{-p/2}}
\underbrace{\sup_{{\bf z} \in \cC}\bP\left( {\bf z}+S(n-m)\in \llbracket {\bf y},  {\bf y}  + \delta {\bf 1}\rrbracket\right)}_{\preceq n^{-d/2}}
\notag\\
 &\preceq n^{-p/2-d/2}\mu\{{\bf z}\in  \cC: \vert {\bf z}\vert \geq A\}.
 \end{align}
 Hence, 
by  combining  (\ref{estimP1}) and (\ref{estimP2}) and by using   Proposition \ref{exitTime},  we obtain
\[x
n^{p/2+d/2}\bP(\tau({\bf x})>n,   {\bf x} +S(n)\in  \llbracket {\bf y}, {\bf y}+\delta{\bf 1}\rrbracket )\preceq e^{-cA^2/2 }+\mu\{{\bf z}\in  \cC: \vert {\bf z}\vert \geq A\} \longrightarrow 0 \ {\rm as} \ A \to +\infty.
\]
 
 \noindent \underline{Proof of Step 2.}  Let ${\bf x} \in \mathcal C,     {\bf y}  \in \cC^{(2)}$,  fix $\delta\in ]0,  1]$ and set $\delta \mathbb Z^d= \delta \bZ e_1\oplus\ldots \oplus \delta \bZ e_d$.
The Markov property and Proposition  \ref{keyinclusionsGENERALversion}  yield
   \begin{align*}
 \bP( \tau({\bf x})>n,   &{\bf x} +S(n)\in  \llbracket {\bf y}, {\bf y}+\delta{\bf 1}\rrbracket) 
\\
 & \leq \sum_{{\bf z}\in  \delta \mathbb Z^d}
\bP( \tau({\bf x})>n,   {\bf x} +S(m) \in\llbracket {\bf z},   {\bf z}+\delta {\bf 1}\rrbracket,   {\bf x} +S(n)\in \llbracket {\bf y},  {\bf y}  + \delta {\bf 1}\rrbracket)
\end{align*}
\begin{align*}
&\qquad  \leq \sum_{{\bf z}\in  \delta \mathbb Z^d}
\bP( \tau({\bf x})>m,   {\bf x} +S(m) \in\llbracket {\bf z},   {\bf z}+\delta {\bf 1}\rrbracket)\\
&\qquad \qquad \qquad \qquad 
\times \sup_{{\bf z}'\in\llbracket {\bf z},   {\bf z}+\delta {\bf 1}\rrbracket}\bP(\tau( {\bf z}')>n-m,   {\bf z}'+S(n-m) \in \llbracket {\bf y},  {\bf y}  + \delta {\bf 1}\rrbracket)
\\
& \qquad   \preceq
{C({\bf x})\over n^{p/2+d/2}}  \sum_{{\bf z}\in  \delta \mathbb Z^d} 
 \bP(\tilde{\tau}_{-\delta}({\bf y})>n-m,  {\bf y} - S(n-m) \in\llbracket {\bf z}-\delta{\bf 1},   {\bf z}+\delta {\bf 1}\rrbracket)\\
 &\qquad \qquad\qquad \qquad\qquad \qquad {\rm by\ inequalities \ }   \eqref{key1} \ {\rm and}\  \eqref{estimate1}, \  
\\
& \qquad  \preceq
{C({\bf x})\over n^{p/2+d/2}}  \   \bP(\tilde{\tau}_{-\delta}({\bf y})>n-m).
\end{align*}
 It remains to check that 
 $\displaystyle 
 \limsup_{n \to +\infty} \sup_{{\bf y}  \in \cC^{(2)}}\bP(\tilde{\tau}_{-\delta}({\bf y})>n)=o(\epsilon)
 $
  uniformly in ${\bf y}  \in \cC^{(2)}$;  for the sake of completeness,  we present here the argument,   following  the proof of formula (78) in   \cite{DW2015}. Firstly,  
   the tail of the distribution of $\tilde \tau_{-\delta}({\bf y})$ may be compared with  the one of the exit time $\tau^{bm}({\bf y}')$ of the cone $\cC $ for  the standard Brownian motion starting from $ {\bf y} '$, for a suitable choice of $ {\bf y} '$: indeed, it holds,  for some positive constant $c$,  
   \[
   \bP(\tilde \tau_{-\delta}({\bf y})>n-m)\leq \bP(\tau^{bm}({\bf y} +c\epsilon \sqrt{n} {\bf 1} )>n-m)+o(n^{-r})
   \]
   where $r>0$ and $o(n^{-r})$ is uniform in ${\bf y}$. Consequently,  by the scaling property of the Brownian motion (see (79) in \cite{DW2015}),  then the parabolic boundary Harnack principle, 
\begin{align*}
   \sup_{{\bf y}  \in \cC^{(2)}}\bP(\tilde \tau_{-\delta}({\bf y})>n-m) 
   &\leq \sup_{\stackrel{{\bf z} \in \cC}{\vert {\bf z}\vert \leq A,  {\rm dist}({\bf z},  \partial \cC)\leq 2 \epsilon}}\bP(\tau^{bm}({\bf z}+c\epsilon {\bf 1})>1/2) +o(n^{-r})
\\
&\preceq 
\sup_{\stackrel{{\bf z} \in \cC}{\vert {\bf z}\vert \leq A,  {\rm dist}({\bf z},  \partial \cC)\leq 2 \epsilon}} u({\bf z}+c\epsilon {\bf 1})+o(n^{-r})\\
&=\underbrace{\sup_{\stackrel{{\bf z}' \in \cC}{\vert {\bf z}'\vert \leq A',  \ {\rm dist}({\bf z}',  \partial \cC)\leq c' \epsilon}} u({\bf z}')}_{=o(\epsilon)}+o(n^{-r})\quad {\rm for \ some \ constants}\quad A',  c'>0. 
\end{align*}
 This completes the proof of step 2.

 \noindent \underline{Proof of Step 3.}  We fix  ${\bf y}  \in \cC^{(3)}$ and  set $ m_\epsilon=\lfloor \epsilon^3n\rfloor $. Let us decompose 
 $\bP( \tau({\bf x})>n,   {\bf x} +S(n)\in  \llbracket {\bf y}, {\bf y}+\delta{\bf 1}\rrbracket)$
 as
 \begin{align*}
\bP( \tau({\bf x})>n,  &  \  {\bf x} +S(n)\in  \llbracket {\bf y}, {\bf y}+\delta{\bf 1}\rrbracket)\\
 & = \sum_{{\bf z}\in  \delta \mathbb Z^d}
\bP( \tau({\bf x})>n,   {\bf x} +S(n-m_\epsilon) \in\llbracket {\bf z},   {\bf z} +\delta{\bf 1}\llbracket,   {\bf x} +S(n)\in  \llbracket {\bf y}, {\bf y}+\delta{\bf 1}\rrbracket)
\\
& = \underbrace{\sum_{\stackrel{{\bf z}\in  \delta \mathbb Z^d}{\vert {\bf y}-{\bf z}\vert > \epsilon \sqrt{n}}}
\bP( \tau({\bf x})>n,   {\bf x} +S(n-m_\epsilon) \in\llbracket {\bf z},   {\bf z} +\delta{\bf 1}\llbracket,   {\bf x} +S(n)\in  \llbracket {\bf y}, {\bf y}+\delta{\bf 1}\rrbracket)}_{\Sigma_1(n,  \epsilon, {\bf y})}
\\
&\qquad +  
 \underbrace{\sum_{\stackrel{{\bf z}\in  \delta \mathbb Z^d}{\vert {\bf y}-{\bf z}\vert \leq\epsilon \sqrt{n}}}
\bP( \tau({\bf x})>n,   {\bf x} +S(n-m_\epsilon) \in\llbracket {\bf z},   {\bf z} +\delta{\bf 1}\llbracket,   {\bf x} +S(n)\in  \llbracket {\bf y}, {\bf y}+\delta{\bf 1}\rrbracket)}
_{\Sigma_2(n,  \epsilon, {\bf y})}
\end{align*}
By the Markov property 
\begin{align*}
&\Sigma_1(n,  \epsilon, {\bf y})  \\
&\ = \sum_{\stackrel{{\bf z}\in  \delta \mathbb Z^d}{\vert {\bf y}-{\bf z}\vert  >\epsilon \sqrt{n}}}
\int_{\llbracket {\bf z},   {\bf z} +\delta{\bf 1}\llbracket }
 \bP \Bigl( \tau({\bf x})>n-m_\epsilon,   {\bf x} +S(n-m_\epsilon) \in {\rm d} {\bf z}' \Bigr)\ \bP(\tau( {\bf z}')>m_\epsilon,   {\bf z}'+S(m_\epsilon) \in  \llbracket {\bf y}, {\bf y}+\delta{\bf 1}\rrbracket)
\\
&\ \leq \sum_{\stackrel{{\bf z}\in  \delta \mathbb Z^d}{\vert {\bf y}-{\bf z}\vert  >\epsilon \sqrt{n}}}
 \bP \Bigl( \tau({\bf x})>n-m_\epsilon,   {\bf x} +S(n-m_\epsilon) \in\llbracket {\bf z},   {\bf z}+\delta {\bf 1}\rrbracket\Bigr)\\
 &\qquad \qquad \qquad \qquad \qquad \qquad \qquad \qquad \times 
 \underbrace{\sup_{\stackrel{{\bf z}'\in\llbracket {\bf z},   {\bf z}+\delta {\bf 1}\rrbracket}{\vert {\bf z}'-{\bf y}\vert >\epsilon\sqrt{n}-\delta}}\bP(\tau( {\bf z}')>m_\epsilon,    {\bf z}'+S(m_\epsilon) \in \llbracket {\bf y},  {\bf y}  +\delta {\bf 1}\rrbracket)}
 _{\preceq n^{-d/2} \epsilon ^{-3d/2} e^{-c/\epsilon},  \ {\rm by\ inequality \  (\ref{estimate}(b)).  }}. 
 \end{align*}
  Hence,  by Proposition \ref{exitTime}, 
\begin{align}\label{sigma_1}
\Sigma_1(n,  \epsilon, {\bf y})&   \preceq  n^{-d/2} \epsilon^{-3d/2} e^{-c/\epsilon} 
\underbrace{\sum_{{\bf z}\in  \delta \mathbb Z^d }
\bP \Bigl( \tau({\bf x})>n-m_\epsilon,   {\bf x} +S(n-m_\epsilon) \in\llbracket {\bf z},   {\bf z}+\delta {\bf 1}\rrbracket \Bigr)}_{=\ \ \bP ( \tau({\bf x})>n-m_\epsilon)   } \notag \\
&\preceq {V({\bf x})\over n^{d/2+p/2}}\epsilon^{-3d/2} e^{-c/\epsilon}. 
\end{align}
 Hence,  $\displaystyle \lim_{\epsilon\to 0}\limsup_{n \to +\infty} n^{d/2+p/2}\sup_{{\bf y} \in \mathcal C^{(3)}}\Sigma_1(n,  \epsilon, {\bf y})=0. $

 Similarly,  by the Markov property, 
\begin{align}\label{decompsigma_2}
&\Sigma_2(n,  \epsilon, {\bf y}) \notag \\
&\ = \sum_{\stackrel{{\bf z}\in  \delta \mathbb Z^d}{\vert {\bf y}-{\bf z}\vert  \leq\epsilon \sqrt{n}}}
\int_{\llbracket {\bf z},   {\bf z} +\delta{\bf 1}\llbracket }
 \bP \Bigl( \tau({\bf x})>n-m_\epsilon,   {\bf x} +S(n-m_\epsilon) \in {\rm d} {\bf z}' \Bigr)\ \bP(\tau( {\bf z}')>m_\epsilon,   {\bf z}'+S(m_\epsilon) \in   \llbracket {\bf y}, {\bf y}+\delta{\bf 1}\rrbracket)
  \notag
 \\
 &\ =  \Sigma'_2(n,  \epsilon, {\bf y})  - \Sigma_2''(n,  \epsilon, {\bf y}) 
 \end{align}
 where
 \[ \Sigma'_2(n,  \epsilon, {\bf y})  :=  \sum_{\stackrel{{\bf z}\in  \delta \mathbb Z^d}{\vert {\bf y}-{\bf z}\vert  \leq\epsilon \sqrt{n}} }
 \int_{\llbracket {\bf z},   {\bf z} +\delta{\bf 1}\llbracket }
 \bP \Bigl( \tau({\bf x})>n-m_\epsilon,   {\bf x} +S(n-m_\epsilon) \in   {\rm d} {\bf z}'\Bigr) \bP(  {\bf z}'+S(m_\epsilon) \in  \llbracket {\bf y}, {\bf y}+\delta{\bf 1}\rrbracket) \  
\]
and 
\begin{align*}
&\Sigma''_2(n,  \epsilon, {\bf y}) \\
&:= \sum_{\stackrel{{\bf z}\in  \delta \mathbb Z^d}{\vert {\bf y}-{\bf z}\vert  \leq\epsilon \sqrt{n}} }
 \int_{\llbracket {\bf z},   {\bf z} +\delta{\bf 1}\llbracket }
 \bP \Bigl( \tau({\bf x})>n-m_\epsilon,   {\bf x} +S(n-m_\epsilon) \in   {\rm d} {\bf z}'\Bigr) \bP(\tau( {\bf z}')\leq m_\epsilon,    {\bf z}'+S(m_\epsilon) \in  \llbracket {\bf y}, {\bf y}+\delta{\bf 1}\rrbracket).
\end{align*}
Since ${\bf y}  \in \cC^{(3)}$ and  $\vert {\bf y}-{\bf z}\vert  \leq\epsilon \sqrt{n}$,  it holds,  for any $ {\bf z}' \in\llbracket {\bf z},   {\bf z}+\delta {\bf 1}\rrbracket$ and $n$ large enough 
 \[
 {\rm dist}( {\bf z}',  \partial \cC)\geq {\rm dist}({\bf y},   \partial \cC)-\vert {\bf y}-{\bf z}\vert -\delta\geq \epsilon \sqrt{n}-\delta\geq  {\sqrt{m_\epsilon}\over 2\sqrt{\epsilon}} ;\]  
 hence,  by (\ref{estimate}($c$)),  
 \[
 \bP(\tau( {\bf z}')\leq m_\epsilon,    {\bf z}'+S(m_\epsilon) \in   {\bf y}  + {\bf B} )\preceq  m_\epsilon^{-d/2}  e^{-c /\epsilon }=n^{-d/2} \epsilon^{-3d/2} e^{-c /\epsilon }.
  \]
 Consequently,
\begin{align}\label{sigma'_2}
 \Sigma''_2(n,  \epsilon, {\bf y})
 &\preceq n^{-d/2}
 \Bigl(
 \sum_{\stackrel{{\bf z}\in  \delta \mathbb Z^d}{\vert {\bf y}-{\bf z}\vert  \leq\epsilon \sqrt{n}}}
 \bP ( \tau({\bf x})>n-m_\epsilon,   {\bf x} +S(n-m_\epsilon) \in\llbracket {\bf z},   {\bf z} +\delta{\bf 1}\llbracket)  \Bigr) \epsilon^{-3d/2} e^{-c /\epsilon }\notag
 \\
 &\leq 
 n^{-d/2}  \ \bP( \tau({\bf x})>n- m_\epsilon)\ \epsilon^{-3d/2} e^{-c /\epsilon },
\end{align}
which readily implies $\displaystyle \lim_{\epsilon\to 0}\limsup_{n \to +\infty}\  n^{d/2+p/2}  \sup_{{\bf y}  \in \cC^{(3)}} \Sigma_2''(n,  \epsilon, {\bf y})=0. $
 
  It remains to control  the term $ \Sigma'_2(n,  \epsilon, {\bf y})$. Using the local central  limit theorem for unconditioned and aperiodic multidimensional random walks  \cite{stone1965},   we have,  uniformly in $\bf y$, 
   
\begin{align}\label{sigma'_2*}
 \Sigma'_2(n,  \epsilon, {\bf y})&=\sum_{\stackrel{{\bf z}\in  \delta \mathbb Z^d}{\vert {\bf y}-{\bf z}\vert  \leq\epsilon \sqrt{n}}}
\int_{\llbracket {\bf z},   {\bf z} +\delta{\bf 1}\llbracket } \  \bP ( \tau({\bf x})>n-m_\epsilon,   {\bf x} +S(n-m_\epsilon) \in {\rm d} {\bf z}')
\bP({{\bf z}'}+S(m_\epsilon) \in  \llbracket {\bf y}, {\bf y}+\delta{\bf 1}\rrbracket)\
\notag
\\ 
&  
=\sum_{\stackrel{{\bf z}\in  \delta \mathbb Z^d}{\vert {\bf y}-{\bf z}\vert  \leq\epsilon \sqrt{n}}}
\int_{\llbracket {\bf z},   {\bf z} +\delta{\bf 1}\llbracket }\bP( \tau({\bf x})>n-m_\epsilon,    {\bf x} +S(n-m_\epsilon) \in {\rm d} {\bf z}')
\notag\\
&\qquad \qquad \qquad \qquad \qquad \qquad  {1\over (2\pi  m_\epsilon)^{d/2}}
 \left( \int_{ \llbracket {\bf y}, {\bf y}+\delta{\bf 1}\rrbracket} 
e^{-\vert{\bf y}'- {\bf z}'\vert^2/2 m_\epsilon}  {\rm d}{\bf y}'\right) (1 +o_n(1)) \notag\\
&\qquad \qquad \qquad \qquad \qquad \qquad {\rm with} \ o_n \ {\rm uniform \ in}\ {\bf y},   {\bf z},  \epsilon
\notag\\
&  
=\sum_{\stackrel{{\bf z}\in  \delta \mathbb Z^d}{\vert {\bf y}-{\bf z}\vert  \leq\epsilon \sqrt{n}}}
\int_{\llbracket {\bf z},   {\bf z} +\delta{\bf 1}\llbracket }\bP ( \tau({\bf x})>n-m_\epsilon,   {\bf x} +S(n-m_\epsilon) \in {\rm d} {\bf z}')
\notag\\
&\qquad \qquad \qquad \qquad \qquad \qquad  {1\over (2\pi  m_\epsilon)^{d/2}}
e^{-\vert {\bf y}- {\bf z}'\vert^2/2 m_\epsilon}\  \delta^d \ (1 +o_n(1))
\notag \\
&\qquad \qquad \qquad \qquad \qquad \qquad {\rm since} \ \vert {\bf y}-{\bf z}\vert  \leq\epsilon \sqrt{n} \ {\rm and } \  \vert {\bf y}'-{\bf y}\vert,   \vert  {\bf z}'-{\bf z}\vert \ {\rm bounded   } 
\notag\\
&  
={\delta^d\over (2\pi  m_\epsilon)^{d/2}} \ (1 +o_n(1))\notag\\
&\qquad \qquad \qquad\times 
\sum_{\stackrel{{\bf z}\in  \delta \mathbb Z^d}{\vert {\bf y}-{\bf z}\vert  \leq\epsilon \sqrt{n}}}
\int_{\llbracket {\bf z},   {\bf z} +\delta{\bf 1}\llbracket }e^{-\vert {\bf y}- {\bf z}'\vert^2/2 m_\epsilon} \bP ( \tau({\bf x})>n-m_\epsilon,   {\bf x} +S(n-m_\epsilon) \in {\rm d} {\bf z}'). 
\end{align}
  From the limit theorem for $(S(n))_n$ conditioned to stay in $\cC$ (see Proposition \ref{exitTime}),  it follows  that,  for every fixed $\epsilon >0$,   as $n \to +\infty$, 
\begin{align}\label{vbuzet}
 \sup_{{\bf y}  \in \cC^{(3)}}&\Big\vert
\sum_{\stackrel{{\bf z}\in  \delta \mathbb Z^d}{\vert {\bf y}-{\bf z}\vert  \leq\epsilon \sqrt{n}}}
\int_{\llbracket {\bf z},   {\bf z} +\delta{\bf 1}\llbracket }\bE \left(  e^{-\vert {\bf y}-{\bf x}-S(n-m_\epsilon)\vert ^2/2m_\epsilon}\Big\vert \tau({\bf x})>n-m_\epsilon\right) 
\notag\\
&
 \qquad\qquad\qquad\qquad -H_0\int_{\vert\sqrt{1-\epsilon^2} r-{{\bf y}\over \sqrt{n}}\vert<\epsilon }u(r) e^{-\vert r\vert^2/2}
e^{-\vert {\bf y}/\sqrt{n}-  \sqrt{1-\epsilon^2} r\vert ^2/2\epsilon^2}{\rm d}r
\Big\vert\longrightarrow 0.
\end{align} 
Since this last integral  equals $u({\bf y}/\sqrt{n}) e^{-\vert{\bf y}\vert^2/2n}(2\pi \epsilon)^{3d/2}+o(\epsilon^{3d/2})$ (see \cite{DW2015} for the details),   we obtain,  combining (\ref{sigma'_2*}) and  (\ref{vbuzet}), 
\begin{equation}\label{sigma''_2}
\lim_{\epsilon \to 0}\limsup_{n \to +\infty}\sup_{{\bf y}  \in \cC^{(3)}} \Big\vert
n^{d/2+p/2} \Sigma'_2(n,  \epsilon, {\bf y})- \kappa_0 H_0\ V({\bf x})\  u({\bf y}/\sqrt{n}) e^{-\vert {\bf y}\vert^2/2n}\ {\rm Leb}(\llbracket {\bf y}, {\bf y}+\delta{\bf 1}\rrbracket) 
\Big\vert=0 
\end{equation} 
for some constant $c>0$  . Step 3 is complete,  combining (\ref{sigma_1}),  
(\ref{decompsigma_2}),  (\ref{sigma'_2}) and  (\ref{sigma''_2}).
 $\Box$  
 
\section{\label{sectionproofTheo2} Proof of Theorem \ref{theolocal}}

 {\bf Proof of inequality }\eqref{localmaj}.  This inequality specifies that of  \cite[Lemma 28]{DW2015}, which is valid in the lattice case.  It is now a classical argument, used in several contexts  (classical random walks on $\mathbb R^d$,   products of random matrices, \dots) and based on the decomposition of the trajectory $(S(k)_{0\leq k\leq  n}$ into three parts; in the non lattice case,  as far as we know it is only done  in dimension 1 (see for instance Proposition 2.3 in \cite{ABKV}). The same argument holds in higher dimension and for the sake of completeness,  we recall it in detail  here. 
 
 We fix ${\bf x} \in \cC$ and a bounded Borel  set $\bf B\subset \cC$  whose boundary is negligible with respect to the Lebesgue measure.  Inequality  \eqref{localmaj} is quite rough   and we fix here  $\delta>0$  such that   ${\bf B}\subset \llbracket {\bf 0},  \delta {\bf 1}\rrbracket$.  

For any $n\geq 1$ and $1\leq k \leq n$,  we set  $S(k,  n):= X_{k+1}+\ldots +X_{n} $ for any $1\leq k\leq n$. We also decompose $S(n)$ into three parts;  writing $n'= \lfloor n/3\rfloor$ and $n''=n-\lfloor n/3\rfloor$,  it holds 
\[ 
S(n)=\underbrace{S(n')}_{S'(n)} \quad +\quad \underbrace{S(n'')-S(n')}_{S''(n):= X_{n'+1}+\cdots+ X_{n''}}  \quad +\quad \underbrace{S(n)-S(n'')}_{S'''(n):= X_{n''+1}+\cdots+ X_{n}}.
\]
Let us fix  ${\bf x},  {\bf y}  \in \cC$.  The event 
 $
E_n =E_n({\bf x},  {\bf y},   {\bf B}) :=\Bigl(\tau({\bf x})  > n,  {\bf x} +S(n)\in  \llbracket {\bf y}, {\bf y}+\delta{\bf 1}\rrbracket\Bigr)
$ is included in the intersection  $E'_n\ \cap \ E''_n\ \cap \ E'''_n$,   with

$\bullet\quad E'_n=\Bigl(\tau({\bf x})  > n'\Bigr);$

$\bullet \quad E''_n =\Bigl(S''(n) \in {\bf y}-{\bf x}-S'(n)-S'''(n) +{\bf B}\Bigr)$;

$\bullet \quad E'''_n = \Bigl( {\bf y}-(X_{k+1}+\cdots + X_n) \in \cC_{-\Delta}\ {\rm for} \ n''\leq k\leq n\Bigr)$.

 The events  $E'_n$ and $E'''_n$ are  both   measurable with respect to the $\sigma$-field  $\mathcal G_n$ generated by $X_1,  \ldots,  X_{n'}$ and $X_{n''+1},  \ldots,  X_n$. Consequently, 
\begin{align*}
\mathbb P(E_n)&\leq \mathbb P\Bigl(\mathbb P \Bigl(E'_n\ \cap \ E''_n\ \cap \ E'''_n \ \vert \ \mathcal G_n \Bigr)\Bigr)
\\
&=\mathbb E \big[ {\bf 1}_{E'_n\ \cap \ E'''_n  } \mathbb P( E''_n \ \vert \ \mathcal \cG_n ) \big].
\end{align*}
The random variable $S''(n)$ is  independent of  $\cG_n$ and its distribution coincides with the one of $ S(n')$. Therefore,  by the classical local central  limit theorem for   random walks on $\bR^d$,   there exists a constant $ C_{\bf B}>0$ such that 
\begin{align*}
\mathbb P \Bigl( E''_n \ \vert \ \mathcal \cG_n \Bigr)&
{\stackrel{\mathbb P {\rm -a.s.}}{\leq}} 
\sup_{{\bf z}\in  \bR^d} \mathbb P\Bigl(S(n')  \in   {\bf z}+{\bf B}\Bigr) \leq {C_{\bf B}\over n^{d/2}}.
\end{align*}
 Since the events  $E'_n$ and $E'''_n$ are independent,   it follows that 
 \begin{equation}\label{probEn}
 \mathbb P(E_n)\leq {C_{\bf B}\over n^{d/2}}\mathbb P(E'_n\ \cap \ E_n''')= {C_{\bf B}\over n^{d/2}}\ \mathbb P(E'_n)\mathbb P(E_n''').
 \end{equation}
The probability of $E'_n$ is controlled by Proposition \ref{exitTime}:   uniformly in ${\bf x} \in \cC$,  
 \begin{equation}\label{probE'n}
 \mathbb P(E'_n) \preceq {1+ \vert {\bf x}\vert^p\over n^{p/2}}.
 \end{equation} 
To control the probability of the event $E_n''' $,   we  notice   that
$\displaystyle 
(X_{n''+1},   \cdots,   X_{n} )  
{\stackrel{ {\rm dist}}{=}}
(X_1,   \cdots,  X_{n'})
$,  so that,   
  uniformly in $\bf {\bf y}\in \cC$,   
\begin{align}\label{probE'''n}
\mathbb P(E_n''') &\leq \mathbb P\Bigl( {\bf y}- S(1) \in \cC_{-\delta},  \cdots,  {\bf y}- S(n') \in \cC_{-\Delta}  \Bigr)\notag\\
 &= \bP(\tilde \tau_{-\delta}({\bf y}) >n)\notag\\
&
  \preceq {1+ \vert {\bf y}\vert^p\over n^{p/2}}. \end{align}
The proof is done,  by combining (\ref{probEn}),  (\ref{probE'n}) and (\ref{probE'''n}).

\noindent {\bf Proof of  convergence}  \eqref{local}. 
We fix several constants $\epsilon >0,    0<\delta < (1-\epsilon)\tilde \delta  $  and $0<t<1$; set $m=\lfloor n/2\rfloor $. 
By the Markov property,  
\begin{align*}
\bP( \tau({\bf x})>n,   {\bf x} +S(n)\in  {\bf B})
  &= \sum_{{\bf z}\in  \delta \mathbb Z^d} \sum_{\stackrel{{\bf y}\in {\tilde \delta}\mathbb  Z^d}{\llbracket {\bf y},  {\bf y}  + \tilde \delta {\bf 1} \llbracket\ \cap \ {\bf B}\neq \emptyset}}
\bP( \tau({\bf x})>n,  {\bf x} +S(m) \in\llbracket {\bf z},   {\bf z}+\delta{\bf 1}\llbracket,  \\
& \qquad     \qquad \qquad \qquad   \qquad \qquad \qquad \qquad \qquad   {\bf x} +S(n)\in \llbracket {\bf y},  {\bf y}  + \tilde \delta {\bf 1} \llbracket\ \cap \ {\bf B})  \\
&=
 \sum_{{\bf z}\in  \delta \mathbb Z^d} \sum_{\stackrel{{\bf y}\in {\tilde \delta}\mathbb  Z^d}{\llbracket {\bf y},  {\bf y}  + \tilde \delta {\bf 1}  \llbracket \ \cap \ {\bf B}\neq \emptyset}}
\int_{\llbracket {\bf z},   {\bf z}+\delta{\bf 1}\llbracket}\bP\Bigl( \tau({\bf x})>n,  {\bf x} +S(m) \in {\rm d}  {\bf z}'\Bigr)\\ 
& \qquad \qquad \qquad \qquad \qquad \bP (\tau( {\bf z}')>n-m,   {\bf z}'+S(n-m)\in
\llbracket {\bf y},  {\bf y}  +\tilde \delta {\bf 1}\llbracket \ \cap \ {\bf B})\
\end{align*}
so that,  by Proposition \ref{keyinclusionsGENERALversion}, 
\begin{equation}\label{doubleineq}
 \Sigma_-(n,  \epsilon, \delta,  \tilde \delta)\leq \bP( \tau({\bf x})>n,   {\bf x} +S(n)\in  {\bf B})\leq \Sigma_+(n, \delta,  \tilde \delta)
\end{equation}
with
\begin{align*}
\Sigma_+(n,  \delta,  \tilde \delta)
&=
 \sum_{{\bf z}\in  \delta \mathbb Z^d} \sum_{\stackrel{{\bf y}\in {\tilde \delta}\mathbb  Z^d}{\llbracket {\bf y},  {\bf y}  + \tilde \delta  {\bf 1}\llbracket\ \cap \ {\bf B}\neq \emptyset}}
\bP(\tau({\bf x})>n,  {\bf x} +S(m) \in\llbracket {\bf z},   {\bf z}+\delta {\bf 1}\rrbracket)\\
& \qquad \qquad \qquad \qquad  \qquad\qquad \times  \bP (\tilde{\tau}_{-\tilde \delta}({\bf y}) >n-m,  {\bf y}-S(n-m)\in
\llbracket {\bf z}-\tilde \delta {\bf 1},   {\bf z}+  \delta {\bf 1}\rrbracket) 
\end{align*}
 and 
\begin{align*}
&\Sigma_-(n,    \epsilon, \delta,  \tilde \delta)
 \\=
& \sum_{{\bf z}\in  \delta \mathbb Z^d \cap \mathcal C } \ \ 
 \sum_ { \stackrel{{\bf y}\in {\tilde \delta}\mathbb  Z^d\cap \mathcal C_{(1-\epsilon)\tilde \delta}}{\llbracket {\bf y},   {\bf y}  +(1-\epsilon)\tilde \delta {\bf 1}  \llbracket \subset {\bf B}}}
\Bigl(\bP(\tau({\bf x})>n,  {\bf x} +S(m) \in\llbracket {\bf z},   {\bf z}+ (1-\epsilon) \delta {\bf 1}\rrbracket)
\\
&   \qquad  \qquad \qquad\times  \bP (\tilde{\tau}_{\tilde \delta}({\bf y})>n-m,  {\bf y} - S(n-m)\in
 \llbracket {\bf z}-(1-\epsilon)(\tilde \delta-\delta){\bf 1}, {\bf z}+(1-\epsilon) \delta {\bf 1}  \rrbracket)\Bigr).
\end{align*}
\underline{(1) Let us first deal with  $\Sigma_+(n,  \delta,  \tilde \delta)$}.  We fix $A>0$ and decompose $\Sigma_+(n,  \delta,  \tilde \delta)$  as
\[
\Sigma_+(n,  \delta,  \tilde \delta)= \Sigma_+^{(1)}(A,  n,  \delta,  \tilde \delta)+\Sigma_+^{(2)}(A,  n,  \delta,  \tilde \delta)
\]
where
\begin{align*}
\bullet \ \Sigma_+^{(1)}(A,  n,  \delta,  \tilde \delta)
&:=  
 \sum_{\stackrel{{\bf z}\in  \delta \mathbb Z^d}{\vert {\bf z}\vert >A\sqrt{n}}}
 \ \  \sum_{\stackrel{{\bf y}\in {\tilde \delta}\mathbb  Z^d}{\llbracket {\bf y},  {\bf y}  + \tilde \delta  {\bf 1}\llbracket\ \cap \ {\bf B}\neq \emptyset}}
\bP(\tau({\bf x})>n,  {\bf x} +S(m) \in\llbracket {\bf z},   {\bf z}+\delta {\bf 1}\llbracket)\notag\\
& \qquad \qquad \qquad \qquad  \qquad\qquad \times  \bP (\tilde{\tau}_{-\tilde \delta}({\bf y}) >n-m,  {\bf y}-S(n-m)\in
\llbracket {\bf z},   {\bf z}+\tilde \delta{\bf 1}\llbracket]
\end{align*}
and
\begin{align*}
\bullet \ \Sigma_+^{(2)}(A,  n,  \delta,  \tilde \delta)
&:=  
 \sum_{\stackrel{{\bf z}\in  \delta \mathbb Z^d}{\vert {\bf z}\vert \leq A\sqrt{n}}} \ \ \sum_{\stackrel{{\bf y}\in {\tilde \delta}\mathbb  Z^d}{\llbracket {\bf y},  {\bf y}  + \tilde \delta  {\bf 1}\llbracket\ \cap \ {\bf B}\neq \emptyset}}
\bP(\tau({\bf x})>n,  {\bf x} +S(m) \in\llbracket {\bf z},   {\bf z}+\delta {\bf 1}\llbracket)\notag\\
& \qquad \qquad \qquad \qquad  \qquad\qquad \times  \bP (\tilde{\tau}_{-\tilde \delta}({\bf y}) >n-m,  {\bf y}-S(n-m)\in
\llbracket {\bf z},   {\bf z}+\tilde \delta{\bf 1}\llbracket)\\
&=  
\sum_{\stackrel{{\bf y}\in {\tilde \delta}\mathbb  Z^d}{\llbracket {\bf y},  {\bf y}  + \tilde \delta  {\bf 1}\llbracket\ \cap \ {\bf B}\neq \emptyset}}
\Sigma_+^{(2)}(A,  n,  \delta,  \tilde \delta,  {\bf y})
\end{align*}
with 
\begin{align*}
\Sigma_+^{(2)}(A,  n,  \delta,  \tilde \delta,  {\bf y})&:= 
\sum_{\stackrel{{\bf z}\in  \delta \mathbb Z^d}{\vert {\bf z}\vert \leq A\sqrt{n}}} 
\bP(\tau({\bf x})>n,  {\bf x} +S(m) \in\llbracket {\bf z},   {\bf z}+\delta{\bf 1}\llbracket)\\
&\qquad \qquad \qquad \qquad   \times  \bP (\tilde{\tau}_{-\tilde \delta}({\bf y}) >n-m,  {\bf y}-S(n-m)\in
\llbracket {\bf z},   {\bf z}+\tilde \delta{\bf 1}\llbracket).
\end{align*}
 
 \underline{$\bullet$ Control of the term $\Sigma_+^{(1)}(A,  n,  \delta,  \tilde \delta)$}
 
According to  Proposition \ref{exitTime} and Proposition \ref{keyinclusionsGENERALversion},  for any $A>0$,  
\begin{align*}
\label{sigma_+'A}
\Sigma_+^{(1)}(A,  n,  \delta,  \tilde \delta)
\leq {C({\bf x}) C({\bf y})\over n^{p+d/2}}
\ \bP \Bigl( \vert S(n-m)\vert \geq  A\sqrt{n} -\vert {\bf y}\vert -\delta-2\tilde \delta\Big\vert \tilde{\tau}_{-\tilde \delta}({\bf y})>n-m\Bigr) 
\end{align*}
so that,    by the limit theorem for $(-S(n))_n$ conditioned to stay in $\cC$, 
\begin{equation}\label{sigma+1Anepislonepsilon'}
\lim_{A\to +\infty} \limsup_{n \to +\infty} \Sigma_+^{(1)}(A,  n,  \delta,  \tilde \delta,  {\bf y})=0.
\end{equation}

 \underline{$\bullet$ Control of the terms $\Sigma^{(2)}_+(A,  n,  \delta,  \tilde \delta)$  } 
 
We fix  ${\bf y}\in {\tilde \delta}\mathbb  Z^d$ such that  $\llbracket {\bf y},  {\bf y}  + \tilde \delta {\bf 1} \llbracket\ \cap \ {\bf B}\neq \emptyset$. 
Applying Theorem \ref{theostonecone} first to $(-S(n))_n$ and then to $(S(n))_n$ we obtain
\begin{align*}
\Sigma^{(2)}_+(A,  n,  \delta,  \tilde \delta,  {\bf y})
&=
c^2 2^{p+d} V({\bf x})V'({\bf y}) \ 
 {1\over n^{p+d}}\ \delta^d\ (\delta+\tilde \delta)^d\\
 &\qquad  \times 
 \underbrace{\left(\sum_{\stackrel{{\bf z}\in  \delta \mathbb Z^d}{\vert {\bf z}\vert \leq A\sqrt{n}}} 
 u\left( {{\bf z}\over \sqrt{m}} \right)  u\left({{\bf z}\over \sqrt{n-m }}\right) \ e^{-{\vert {\bf z}\vert^2\over 2n}}\right)}_{I_n(A,  \delta)}+o(n^{-p-d/2})
\end{align*}
with,  by the homogeneity property (\ref{homogene-u}) satisfied by the function   $u$,  
\[
\lim_{n \to +\infty}\delta^d n^{-d/2}
 I_n(A,  \delta)= 2^p \int_{{\bf z} \in \cC\ \mid \ \vert w\vert \leq A} u^2({\bf z}) e^{-\vert {\bf z}\vert^2/2 }{\rm d}{\bf z}.
\]
Consequently,  
\begin{align*}
\lim_{A\to +\infty} \lim_{n \to +\infty} n^{p+d/2}\ \Sigma^{(2)}_+(A,  n,  \delta,  \tilde \delta,  {\bf y})
&= c^2  2^{2p+d} V({\bf x})V'({\bf y})  
\ (\delta+\tilde \delta)^d \ H_0\\
&\qquad \qquad \qquad {\rm with} \qquad H_0:=  \int_{{\bf z} \in \cC} u^2({\bf z}) e^{-\vert {\bf z}  \vert^2/2 }{\rm d}{\bf z}  
\end{align*}
so that
\begin{align*}
\lim_{\delta\to 0}
\lim_{A\to +\infty} \lim_{n \to +\infty} n^{p+d/2}\ \Sigma^{(2)}_+(A,  n,  \delta,  \tilde \delta,  {\bf y})=
\ c^2\ 2^{2p+d}\ H_0\   V({\bf x})V'({\bf y})   
\  (\tilde \delta)^d.
\end{align*}
Hence,  summing over ${\bf y}\in {\tilde \delta}\mathbb  Z^d$ such that  $\llbracket {\bf y},  {\bf y}  + \tilde \delta {\bf 1}  \llbracket\ \cap \ {\bf B}\neq \emptyset$,  
\begin{equation}\label{sigma+2Anepislonepsilon'}
\lim_{\delta \to 0} \lim_{A\to +\infty} \lim_{n \to +\infty} n^{p+d/2}\ \Sigma^{(2)}_+(A,  n,  \delta,  \tilde \delta )= c^2 \  H_0\ {V({\bf x})\over (t(1-t))^{p +d/2}} 
\underbrace{\sum_{\stackrel{{\bf y}\in {\tilde \delta}\mathbb  Z^d}{\llbracket {\bf y},  {\bf y}  + \tilde \delta {\bf 1} \llbracket\ \cap \ {\bf B}\neq \emptyset}}
\ V'({\bf y})\ (\tilde \delta)^d}_{\int_{\bf B}V'({\bf y}) {\rm d} {\bf y}\ +\ o(\tilde \delta).}
\end{equation}
Finally,   combining  (\ref{sigma+1Anepislonepsilon'}) and (\ref{sigma+2Anepislonepsilon'}) implies
\begin{equation}\label{sigma+}
\lim_{\tilde \delta\to 0}
\lim_{\delta \to 0} \lim_{n \to +\infty} n^{p+d/2}\ \Sigma_+(n,  \delta,  \tilde \delta )= c^2 \ 2^{2p+d}  H_0\  V({\bf x}) 
\int_{\bf B}V'({\bf y}) {\rm d} {\bf y}.\end{equation}
\underline{(2) Term $\Sigma_-(n,  \delta,  \tilde \delta)$.} 

Similarly,  it holds
\begin{equation}\label{sigma-}
\lim_{\tilde \delta\to 0}
\lim_{\delta \to 0} \lim_{\epsilon\to 0} \lim_{n \to +\infty} n^{p+d/2}\ \Sigma_-(n,  \delta,  \tilde \delta )= c^2 \ 2^{2p+d}  H_0\  V({\bf x}) 
\int_{\bf B}V'({\bf y}) {\rm d} {\bf y}.\end{equation}

 The proof is complete,  combining (\ref{doubleineq}),  (\ref{sigma+}) and (\ref{sigma-}).
 $\Box$
 
 \section*{Acknowledgements}
This research work was initiated in 2019 when  CNRS supported the third author D. T. Son. 
M. Peign\'e and  D. C. Pham were  supported  recently by ANR-23-CE40-0008. 
The three authors   thank  the anonymous reviewer for his careful reading and numerous comments, which have helped to improve this text.

 \begin{figure}[h]
\centering
\setlength{\unitlength}{1mm}

\end{figure}


\begin{thebibliography}{99}

\bibitem{ABKV}
\textsc{Afanasyev V. I.,  B\"oinghoff C.,  Kersting G.,  \& Vatutin V. A. } (2012)  {\it Limit theorems for weakly subcritical branching processes in random environment}, Journal of Theoretical Probability,  vol. 25,  no. 3,  703--732.

\bibitem{BS} {\sc  Banuelos, R., \& Smits, R.G.}  (1997) {\it Brownian motion in cones},  Probab. Theory Related Fields, 
vol. 108,  no. 1,  299--319.

 


\bibitem{BBE} {\sc  Babillot M.,  Bougerol Ph.  \&  Elie L.  }  (1997) {\it The random difference equation
    $X_{n+1}=A_nX_n+B_n$
 in the critical case},  Annals of Probability,  
vol. 25,  no. 1,  478--493.

\bibitem{BL}\textsc{Bougerol Ph.  \&   Lacroix J.} (1985) \textit{Products of Random Matrices with Applications to Schr\"odinger  Operators},  Birkh\"auser.

\bibitem{Bour}\textsc{Bourbaki N.} (2007) \textit{ \'El\'ements de math\'ematique. Topologie  g\'en\'erale.  Chapitres  5 \`a 10 },  Berlin: Springer.

\bibitem{DEW}\textsc{Denis D., Elizarov N. \& Wachtel V. } (2026) \textit{ Harmonic polynomials and other exactly computable characteristics for 2-dimensional random walks in cones},  arXiv:2601.03866v1.

\bibitem{DV} \textsc{Dyakonova E. E.  \& Vatutin V. A. } (2017) {\it Multitype branching processes in random environment: survival probability for the critical case},  Teor. Veroyatnost. i Primenen.,  2017,  vol. 62,   no. 4,   634--653.

 
\bibitem{DW2015}\textsc{Denisov D. \&  Wachtel V.}  (2015)  {\it Random walks in cones},  
	The Annals of Probab.,   vol. 43,    no. 3,  992--1044.  

\bibitem{DW2019}\textsc{Denisov D. \&  Wachtel V.}  (2019)  {\it Alternative constructions of a harmonic function for a random walk in a cone},  
	Electron. J. Probab.  vol. 24,    no. 92,  1--26.  


\bibitem{FIM} {\sc  Fayolle,  G.,  Iasnogorodski  R. \& Malyshev V.} (1999)
{\it   Random walks in the quarter-plane: algebraic
methods,  boundary value problems and applications},   Springer-Verlag,  Berlin.



\bibitem{Fu} {\sc  Fulmek  M.} (2012) {\it Viewing Determinants as Nonintersecting Lattice Paths yields Classical Determinantal Identities Bijectively},   The Elect. J. of Combinatoric,   vol. 19,  no. 3,  1--46.
		
	\bibitem{FK}\textsc{Furstenberg H. \&   Kesten H.}  (1960)  {\it Products of random matrices},  
	Ann. Math. Statist.,   vol. 31,   457--469.  
	
	
	\bibitem{GK}\textsc{Geiger J. \&   Kersting G. } (2001) {\it The survival probability of a critical branching process in random environment},   
Theory of Probability and its Applications,   vol. 45,  No. 3,  517--525.
		 
\bibitem{GZ}\textsc{Gessel I.M.  \&   Zeilberger,  D. } (1992) {\it Random walk in a Weyl chamber},   
  Proc. Amer. Math. Soc.,  vol. 115:,  27--31.



	\bibitem{Hennion}\textsc{Hennion H.} (1997) {\it Limit theorems for products of positive random matrices},  Annals of Probability,   vol. 25,  no. 4,   1545--1587.  

	\bibitem{HR}\textsc{Hubert E. \&   Raschel K.}  (1960)  {\it Discrete harmonic polynomials in multidimensional orthants},  
	arXiv:2505.19622v1.
	
	\bibitem{LP1}\textsc{Lepage E. \&  Peign\'e M.} (1997) {\it  A local limit theorem on the semi-direct product of $\mathbb R^{*+}$ and $\mathbb R^d$},   Ann. Inst. Henri Poincar\'e,  vol. 33,  no. 2,  223--252.

\bibitem{LP2}\textsc{Lepage E.  \&  Peign\'e M.} (1999) {\it  Local limit theorems on some non unimodular groups.},    Revista Matematica Iberoamericana,  vol. 15,  no. 1,  117--141.
	
		 
\bibitem{LPP}\textsc{Le Page E.,  Peign\'e M. \& Pham C.} (2018) {\it The survival probability of a critical multitype branching process in i.i.d. random environment},  Ann. Probab. vol. 46,  no. 5,  2946--2972.
	
 \bibitem{Lesigne2021}\textsc{Lesigne E.} (2021) {\it private communication}.
	 
	
 \bibitem{Spitzer} {\sc Spitzer L.}(1964) {\it Principles of random walks},   D. van Nostrand Company. 
 
 

 
\bibitem{stone1965} {\sc Stone C.}(1965) {\it A local limit theorem for nonlattice multidimensional distribution functions},   Ann. Math. Statist.
vol. 36,   no. 2,  546--551.
 
 \bibitem{Var} {\sc Varopoulos  N.Th.}(2000) {\it Potential theory in conical domains. II.},    Math. Proc. Camb. Phil. Soc., 
vol.36,  No.2,  546--551.  

 \bibitem{VW} {\sc  Vatutin V. \&    Wachtel V.}(2009) {\it Local probabilities for random walks conditioned to stay positive},    Probab. Theory Related Fields,  vol. 143,  177--217,  2009.
\end{thebibliography}
\end{document}